\documentclass[reqno,12pt]{amsart}
\usepackage[margin=1in]{geometry}
\usepackage{amsfonts}
\usepackage[all]{xy}
\usepackage{enumitem}
\usepackage{mathtools}
\usepackage{amssymb,amsmath,url}
\usepackage{setspace}
\usepackage{amsthm}
\usepackage{cite}
\usepackage{ }
\usepackage{mathrsfs}
\numberwithin{equation}{section}
\usepackage[normalem]{ulem}
\usepackage{graphicx}
\usepackage{lmodern}
\graphicspath{ {images/} }
\usepackage{pdfpages}
\usepackage{MnSymbol}

	\author{Zhenghui Huo and Brett D. Wick}
\title[$L^p$ regularity of the Bergman projection on the symmetrized polydisc]{$L^p$ regularity of  the Bergman projection on the symmetrized polydisc}
\begin{document}
	\thanks{BDW's research is partially supported by National Science Foundation grants DMS \# 1800057 and \# 2054863, and \#	2000510 and Australian Research Council grant DP 220100285. ZH's research is partially supported by National Science Foundation of China NSFC Grant \# 12201265}.
		\newtheorem{thm}{Theorem}[section]
	\newtheorem{cl}[thm]{Claim}
	\newtheorem{lem}[thm]{Lemma}
		\newtheorem{ex}[thm]{Example}
	\newtheorem{de}[thm]{Definition}
		\newtheorem{prop}[thm]{Proposition}
	\newtheorem{co}[thm]{Corollary}
	\newtheorem*{thm*}{Theorem}
	\theoremstyle{definition}
		\newtheorem{rmk}[thm]{Remark}
			\address{Zhenghui Huo, Zu Chongzhi Center for Mathematics and Computational Sciences, Duke Kunshan University,  Kunshan, 215316, Jiangsu, China}
		\email{zhenghui.huo@duke.edu}
		\address{Brett D. Wick, Department of Mathematics and Statistics, Washington University in St. Louis,  St. Louis, MO 63130-4899, USA}
		\email{wick@math.wustl.edu}
	\maketitle
	\begin{abstract}
We study the $L^p$ regularity of the Bergman projection $P$ over the symmetrized polydisc in $\mathbb C^n$. We give a decomposition of the Bergman projection on the polydisc and obtain an operator equivalent to the Bergman projection over anti-symmetric function spaces. Using it, we obtain the $L^p$ irregularity of $P$ for $p=\frac{2n}{n-1}$ which also implies that $P$ is $L^p$ bounded if and only if  $p\in (\frac{2n}{n+1},\frac{2n}{n-1})$. 
		\medskip
		
		\noindent
		{\bf AMS Classification Numbers}: 32A25, 32A36,  32A50
		
		\medskip
		
	\end{abstract}
\section{Introduction}
Let $\Omega$ be a domain in the complex Euclidean space $\mathbb C^n$. Let $dV$ denote the Lebesgue measure. The Bergman projection $P_\Omega$ is the orthogonal projection from $L^2(\Omega)$ onto the Bergman space $A^2(\Omega)$, the space of all square-integrable holomorphic functions. Associated with $P_\Omega$, there is a unique function $K_\Omega$ on $\Omega\times\Omega$ such that for any $f\in L^2(\Omega)$:
\begin{equation}
P_\Omega(f)(z)=\int_{\Omega}K_\Omega(z;\bar w)f(w)dV(w).
\end{equation}The positive Bergman operator $P^+_{\Omega}$ is given by 
\begin{equation}
P^+_\Omega(f)(z)=\int_{\Omega}|K_\Omega(z;\bar w)|f(w)dV(w).
\end{equation}
By its definition, the Bergman projection is $L^2$ bounded. An active area of research in several complex variables and harmonic analysis considers the $L^p$ regularity of $P_\Omega$ for $p\neq 2$. In particular, people are interested in the connection between the boundary geometry of pseudoconvex domains and the $L^p$ behavior of the projection. On a wide class of domains,  the Bergman projection is $L^p$ regular for all $1<p<\infty$. See for instance \cite{Fefferman,PS,McNeal1,McNeal3,NRSW,McNeal3,McNeal2,MS,CD,EL,BS}. On some other domains, the projection has only a finite range of mapping regularity. See for example \cite{Yunus,DebrajY,EM,EM2,CHEN,Liwei19,CJY,BCEM}. We also refer to \cite{Yunus20} for a survey on the problem.

In this paper, we focus on the Bergman projection on the symmetrized polydisc $\mathbb G^n$. Let $\mathbb D^n$ denote the polydisc in $\mathbb C^n$. Let $\Phi_n$ be the rational holomorphic mapping on $\mathbb C^n$ given by $\Phi_n(w_1,\dots,w_n)=(p_1(w),\dots, p_n(w))$ where
$p_j(w)$ is the symmetric polynomial in $w$ of degree $j$:
\[p_j(w_1,w_2,\dots,w_n)=\sum_{k_1<k_2<\cdots<k_j}w_{k_1}w_{k_2}\cdots w_{k_j}.\]
 The symmetrized polydisc $\mathbb G^n$ is the image of $\mathbb D^n$ under $\Phi_n$:
	\begin{equation}\label{1.30}\mathbb G^n:=\{(p_1(w),\dots,p_n(w)):w\in \mathbb D^n\}.\end{equation}	
When $n=2$, the symmetrized bidisc \begin{equation}\label{1.31}\mathbb G:=\mathbb G^2=\{(w_1+w_2,w_1w_2):(w_1,w_2)\in \mathbb D^2\}
\end{equation} serves as an interesting example in several complex variables. It is a first known example of many phenomena. We list some of them here below: 
\begin{itemize}
	\item the Lempert theorem may hold on bounded pseudoconvex domains that are not biholomorphically equivalent to any convex domains. \cite{AY} 
	\item bounded $\mathbb C$-convex domains are not necessarily biholomorphically equivalent to convex ones. \cite{Nikolov_Pflug_Zwonek_2008}
\end{itemize} 
See also \cite{agler_young_2000,Sarkar,AGLER2018} for some recent work on $\mathbb G$.

In addition, the symmetrized polydisc $\mathbb G^n$ also serves as an example of  a quotient domain and is biholomorphically equivalent to $\mathbb D^n/\mathcal S_n$ where $\mathcal S_n$ is the group of permutations of $n$ coordinate variables in $\mathbb C^n$. See \cite{Ghosh21, Dall'AraMonguzzi} for some recent studies regarding Bergman projections over quotient domains of the form $\Omega/G$.

Partially due to $\mathbb G^2$'s interesting properties, the $L^p$ regularity of $P_{\mathbb G^2}$ and $P_{\mathbb G^n}$ has also attracted attention in recent years. In \cite{Liwei19}, Chen, Krantz, and Yuan showed that $P_{\mathbb G^n}$ is $L^p$ bounded for $p\in (1+\frac{n-1}{\sqrt{n^2-1}},1+\frac{\sqrt{n^2-1}}{n-1})$. Later, Chen, Jin, and Yuan \cite{CJY} improved the $L^p$ regular range of  $P_{\mathbb G}$ to $({4}/{3},4)$ and established the Sobolev estimates for $P_{\mathbb G}$. While preparing this paper, the authors were informed of a discrepancy between the arXiv version of \cite{CJY} and the version those authors submitted to a journal for publication. In a recent update of \cite{CJY} posted to the arXiv, the range of $L^p$ regularity for the symmetrized polydisc is at least $(\frac{2n}{n+1},\frac{2n}{n-1})$, see \cite[Remark 1.5]{CJY23}. 
The main idea in the proof of these results is to use  Bell's transformation formula \cite{Bell} to reformulate the $L^p$ regularity problem of $P_{\mathbb G^n}$  into a weighted $L^p$ regularity problem of $P_{\mathbb D^n}$ over a weighted $L^p$ space of anti-symmetric functions. Yet, the precise $L^p$ regular range  for $P_{\mathbb G^n}$ was not previously known.

There are mainly two challenges on obtaining the sharp $L^p$ estimates of $P_{\mathbb G^n}$: 1. the complexity of  the Jacobian of $(p_1,\dots,p_n)$ for large $n$ dimension makes estimations complicated. 2. the cancellation caused by integrating anti-symmetric functions creates obstacles to precisely analyze the (un)boundedness of the operator.  To us, the second issue is more crucial and distinguishes the problem on $\mathbb G^n$ from other settings like the Hartogs triangle. Actually, this issue leads to an interesting yet nontrivial weighted inequality problem in harmonic analysis. We elaborate below using a simple analogical example:

\textit{Let $T$ be a singular integral operator on $L^p(\mathbb R^2)$. Set $$L^p_{\text{anti}}(\mathbb R^2,|x_1-x_2|^a):=\{f\in L^p(\mathbb R^2,|x_1-x_2|^a):f(x_1,x_2)=-f(x_2,x_1)\}.$$ For which $p$ is the operator $T$ bounded on $L^p_{\text{anti}}(\mathbb R^2,|x_1-x_2|^a)$?}

From the classical weighted theory, the singularity of the weight function $|x_1-x_2|^a$ over the line $\{x_1=x_2\}$ may cause unboundedness issue for $T$ over $L^p(\mathbb R^2,|x_1-x_2|^a)$. On the other hand, the antisymmetry property $f(x_1,x_2)=-f(x_2,x_1)$ implies that  for any $U\subseteq \mathbb R$ $$\int_{U\times U}fdV=0,$$ suggesting possible better behavior of $T$  on the subspace $L^p_{\text{anti}}(\mathbb R^2,|x_1-x_2|^a)$ than on the entire weighted $L^p$ space. Nevertheless, the usual harmonic analysis methods for weighted $L^p$ cannot be directly applied to this subspace case. 

In this paper, we overcome these issues on $\mathbb G^n$ and give the precise $L^p$ regular range for $P_{\mathbb G^n}$ and $P^+_{\mathbb G^n}$:
\begin{thm}\label{main}
$P_{\mathbb G^n}$ and  $P^+_{\mathbb G^n}$  are $L^p$ bounded if and only if $p\in (\frac{2n}{n+1},\frac{2n}{n-1})$.
\end{thm} 
When $n=2$,  $P_{\mathbb G}$  is $L^p$ bounded if and only if $p\in (\frac{4}{3},4)$.  In contrast to this result,   Dall'Ara and Monguzzi \cite{Dall'AraMonguzzi}  recently showed that, if one replaces $\mathbb D^2$ by unit ball $\mathbb B_2$ in (\ref{1.31}), the Bergman projection over the newly formed domain $\{(w_1+w_2, w_1w_2):(w_1,w_2)\in \mathbb B_2\}$ will possess completely different $L^p$ mapping properties.  In particular, they proved the following:

\vskip 5pt
\textit{Set $D_{2^k}:=\{(w_1^{2^k}+w_2^{2^k}, w_1w_2):(w_1,w_2)\in \mathbb B_2\}$ with $k\in\mathbb N\cup\{0\}$. Then the Bergman projection on $D_{2^k}$ is $L^p$ bounded for all $p\in (1,\infty)$.}

\vskip 5pt
Our computations suggest that the distinction between results on $\mathbb G$ and $D_{2^k}$ is caused by the product structure of $\mathbb D^2$. It is yet to be investigated on what exact geometric property of these domains will determine the $L^p$ mapping behaviors of the projection over them.

Our proof strategy of Theorems \ref{main} can be summarized as follows:
\begin{enumerate}
	\item Similar to \cite{Liwei19,CJY}, we reformulate Theorem \ref{main} into a weighted $L^p$ regularity result of $P_{\mathbb D^n}$ for anti-symmetric functions on the polydisc $\mathbb D^n$ (see Theorems \ref{main1} and  \ref{main10}).
	\item We prove in detail the $L^p$ boundedness results for $p\in(\frac{2n}{n+1},\frac{2n}{n-1})$ using known weighted estimates on the polydisc. (See Theorem \ref{main10}, Section 3, and \cite[Remark 1.5]{CJY23}).
	\item To obtain the unboundedness result  for the case $p=\frac{2n}{n-1}$, we decompose $P_{\mathbb D^n}$ into the sum of two operators $T^n_1$ and $T^n_2$ (see (\ref{1.3}) and (\ref{1.4})) where $T^n_1=0$ and $T^n_2=P_{\mathbb D^n}$ over spaces of anti-symmetric functions (see Lemmas \ref{lem2.3} and \ref{lem4.3}).
	\item By using $T^n_2$, we further reduce the (un)boundedness problem of $P_{\mathbb D^n}$ over a space of anti-symmetric functions into a problem about an operator $\tilde T^n$ over a different space of symmetric functions. Finally, we provide  examples for the unboundedness of $\tilde T^n$ there (see Theorems \ref{main2} and  \ref{thm4.4} and their proofs).
\end{enumerate}
We remark that the decomposition $P_{\mathbb D^n}=T^n_1+T^n_2$ is crucial in our proof. Using the kernel function of  $T^n_2$, we are able to ``cancel out'' part of the weight of the space, transform the problem from an anti-symmetric function space  to a symmetric one, and reduce norm computation difficulty in $n$ dimensional case all at once.

Our paper is organized as follows: In Section 2, we provide known lemmas and reduce $L^p$ estimates of $P_{\mathbb G^n}$ and $P^+_{\mathbb G^n}$ into weighted $L^p$ estimates of $P_{\mathbb D^n}$ for (anti-)symmetric functions.  In Section 3, we recall the known weighted $L^p$ norm estimates of $P_{\mathbb D}$ and give a detailed proof for the $L^p$ boundedness result for $P_{\mathbb G^n}$ and $P^+_{\mathbb G^n}$. In Section 4, we present the decomposition of $P_{\mathbb D^n}$ and examples for the $L^p$ irregularity of $P_{\mathbb G^n}$ for $p=\frac{2n}{n-1}$. In Section 5, we point out some directions for future research.

Given functions of several variables $f$ and $g$, we use $f\lesssim g$ to denote that $f\leq Cg$ for a constant $C$. If $f\lesssim g$ and $g\lesssim f$, then we say $f$ is comparable to $g$ and write $f\approx g$.

\noindent \textbf{Acknowledgements.} Authors would like to thank Gian Maria Dall'Ara, Yuan Yuan, Yuan Zhang,  W\l odzimierz Zwonek, and the anonymous referee for their valuable comments and suggestions.

\section{Pull back from $\mathbb G^n$ to $\mathbb D^n$}
This section focuses on reformulating the $L^p$ regularity of $P_{\mathbb G^n}$ into a problem on the polydisc $\mathbb D^n$.  Most of the lemmas and results were included in \cite{Liwei19,CJY}. We provide proofs here for completeness of our paper.
\subsection{From $\mathbb G^n$ to $\mathbb D^n$} Recall that
$\Phi_n(w)=(p_1(w),p_2(w),\dots,p_n(w))$ where \[p_j(w_1,w_2,\dots,w_n)=\sum_{k_1<k_2<\cdots<k_j}w_{k_1}w_{k_2}\cdots w_{k_j}.\] 
Then $\Phi_n$ is a ramified rational proper covering  map of order $n!$  with  complex holomorphic Jacobian
 $$J_{\mathbb C}\Phi_n=\prod_{j<k}(w_j-w_k).$$ See for example \cite{Liwei19}. 
Let $h\in L^p(\mathbb G^n)$. Via a change of variables,
the estimate
\[\|P_{\mathbb G^n}(h)\|_{L^p(\mathbb G^n)}\lesssim \|h\|_{L^p(\mathbb G^n)}\]
is equivalent to 
\begin{equation}\label{2.1}\|P_{\mathbb G^n}(h)\circ\Phi_n\|_{L^p(\mathbb D^n,|J_{\mathbb C}\Phi_n|^2)}\lesssim \|h\circ\Phi_n\|_{L^p(\mathbb D^n,|J_{\mathbb C}\Phi_n|^2)}.\end{equation}
Using the Bell's transformation formula \cite{Bell}, 
\[P_{\mathbb D^n}(J_{\mathbb C}\Phi_n\cdot(h\circ \Phi_n))=J_{\mathbb C}\Phi\cdot (P_{\mathbb G^n}(h)\circ \Phi_n),\]
(\ref{2.1}) becomes the following weighted estimate:
\begin{equation}\label{1.2}\|P_{\mathbb D^n}(J_{\mathbb C}\Phi_n\cdot(h\circ \Phi_n))\|_{L^p(\mathbb D^n,|J_{\mathbb C}\Phi_n|^{2-p})}\lesssim \|J_{\mathbb C}\Phi_n\cdot h\circ \Phi_n\|_{L^p(\mathbb D^n,|J_{\mathbb C}\Phi_n|^{2-p})}.\end{equation}
By Bell's transformation formula for the Bergman kernel,
\[\sum_{j=1}^{n!}K_{\mathbb D^n}(z; \overline {\phi_j(w)})\overline {J_{\mathbb C}(\phi_j)(w)}=J_{\mathbb C}\Phi_n(z)K_{\mathbb G_n}(\Phi_n(z),w),\]
where $\phi_j$  are the $n!$ local inverses of $\Phi$. Therefore,
to show the estimate  
\[\|P^+_{\mathbb G^n}(h)\|_{L^p(\mathbb G^n)}\lesssim \|h\|_{L^p(\mathbb G^n)},\]
it is sufficient to prove that
\begin{equation}\label{1.20}\|P^+_{\mathbb D^n}(|J_{\mathbb C}\Phi_n|\cdot(h\circ \Phi_n))\|_{L^p(\mathbb D^n,|J_{\mathbb C}\Phi_n|^{2-p})}\lesssim \||J_{\mathbb C}\Phi_n|\cdot h\circ \Phi\|_{L^p(\mathbb D^n,|J_{\mathbb C}\Phi_n|^{2-p})}.\end{equation}

Let $\mathcal S_n$ denote the family of all permutations of $\{z_1,\dots,z_n\}$. Since $\Phi_n$ is invariant  under any permutation, the function $h\circ \Phi_n$ also inherits symmetry properties.  To clearly describe them, we give several definitions below.  For $j,k\in \{1,\dots,n\}$ with $j<k$, we let $\tau_{j,k}$ denote the 2-cycle in $\mathcal S_n$ that interchanges $z_j$ and $z_k$. For $j=1,\dots,n$, we will also abuse the notation for $\tau\in\mathcal S_n$ and let $\tau(j)$ denote the index such that $\tau(z_j)=z_{\tau(j)}$.
\begin{de}\label{de1}
	Let $f$ be a function on $\mathbb D^n$.  \begin{enumerate}
		\item $f$ is called $(j,k)$ symmetric if $f(z_1,\dots,z_n)=f\circ \tau_{j,k}(z_1,\dots,z_n)$, and is called  symmetric if $f(z_1,\dots,z_n)=f\circ \tau_{j,k}(z_1,\dots,z_n)$ for any $j\neq k$.
			\item $f$ is called $(j,k)$ anti-symmetric if $f(z_1,\dots,z_n)=-f\circ \tau_{j,k}(z_1,\dots,z_n)$ and  is called anti-symmetric if $f(z_1,\dots,z_n)=-f\circ \tau_{j,k}(z_1,\dots,z_n)$ for any $j\neq k$.
	\end{enumerate}
\end{de}

By the above definition, $h\circ \Phi_n$ is symmetric while $J_{\mathbb C}\Phi_n$ is anti-symmetric. Therefore,  the function $J_{\mathbb C}\Phi_n\cdot h\circ \Phi_n$ is anti-symmetric and $|J_{\mathbb C}\Phi_n|\cdot h\circ \Phi_n$ is symmetric. It's also not hard to see that $P_{\mathbb D^n}(J_{\mathbb C}\Phi_n\cdot(h\circ \Phi_n))$ and $P^+_{\mathbb D^n}(J_{\mathbb C}\Phi_n\cdot(h\circ \Phi_n))$ are anti-symmetric and $P^+_{\mathbb D^n}(|J_{\mathbb C}\Phi_n|\cdot(h\circ \Phi_n))$ is symmetric. 
Set \begin{align}&L^p_{\text{anti}}(\mathbb D^n,|J_{\mathbb C}\Phi_n|^{2-p}):=\{f\in L^p(\mathbb D^n,|J_{\mathbb C}\Phi_n|^{2-p}):f \text{ is anti-symmetric}\},\\&L^p_{\text{sym}}(\mathbb D^n,|J_{\mathbb C}\Phi_n|^{2-p}):=\{f\in L^p(\mathbb D^n,|J_{\mathbb C}\Phi_n|^{2-p}):f \text{ is symmetric}\}.\end{align}
$L^p_{\text{anti}}(\mathbb D^n,|J_{\mathbb C}\Phi_n|^{2-p})$ and $L^p_{\text{sym}}(\mathbb D^n,|J_{\mathbb C}\Phi_n|^{2-p})$ turn out to be equivalent to $L^p(\mathbb G^n)$.

The next lemma gives the norm equivalence of $L^p_{\text{anti}}(\mathbb D^n,|J_{\mathbb C}\Phi_n|^{2-p})$, $L^p_{\text{sym}}(\mathbb D^n,|J_{\mathbb C}\Phi_n|^{2-p})$, and $L^p(\mathbb G^n)$. When $p=2$, this lemma can be viewed as a special case of  \cite[Theorem 1]{Trybula}.
\begin{lem}\label{lem2.1}
	The following statements are true:
	\begin{enumerate}
		\item $L^p_{\text{anti}}(\mathbb D^n,|J_{\mathbb C}\Phi_n|^{2-p})$ is norm equivalent to $L^p(\mathbb G^n)$ via the mapping:
		\begin{equation}\label{2.6} f\mapsto \sum_{j=1}^{n!}\left(\frac{f}{J_{\mathbb C}\Phi}\right)\circ \phi_j.
		\end{equation}
		\item $L^p_{\text{sym}}(\mathbb D^n,|J_{\mathbb C}\Phi_n|^{2-p})$ is norm equivalent to $L^p(\mathbb G^n)$  via the mapping:
	\begin{equation}f\mapsto \sum_{j=1}^{n!}\left(\frac{f}{|J_{\mathbb C}\Phi_n|}\right)\circ \phi_j.
	\end{equation}
	\end{enumerate}
\end{lem}
\begin{proof}
	We prove the statement for $L^p_{\text{anti}}(\mathbb D^n,|J_{\mathbb C}\Phi_n|^{2-p})$. The proof for $L^p_{\text{sym}}(\mathbb D^n,|J_{\mathbb C}\Phi_n|^{2-p})$ is similar. We begin by showing that the mapping in (\ref{2.6}) is norm preserving. Since $f$ is anti-symmetric, the function $\frac{f}{J_{\mathbb C}\Phi_n}$ is symmetric. Thus, $\left(\frac{f}{J_{\mathbb C}\Phi_n}\right)\circ \phi_j=\left(\frac{f}{J_{\mathbb C}\Phi_n}\right)\circ \phi_k$ for any $j,k$ and 
	\begin{align*}\int_{\mathbb D^n}|f|^p|J_{\mathbb C}\Phi_n|^{2-p}dV=&\int_{\mathbb D^n}\left|\frac{f}{J_{\mathbb C}\Phi_n}\right|^p|J_{\mathbb C}\Phi|^{2}dV\\=&\sum_{j=1}^{n!}\int_{\phi_j(\mathbb G^n)}\left|\frac{f}{J_{\mathbb C}\Phi_n}\right|^p|J_{\mathbb C}\Phi_n|^{2}dV\\=&\sum_{j=1}^{n!}\int_{\mathbb G^n}\left|\left(\frac{f}{J_{\mathbb C}\Phi_n}\right)\circ \phi_j\right|^pdV\\=&(n!)^{1-p}\int_{\mathbb G^n}\left|\sum_{j=1}^{n!}\left(\frac{f}{J_{\mathbb C}\Phi_n}\right)\circ \phi_j\right|^pdV.\end{align*}
Note also that $h\mapsto \frac{1}{n!}J_{\mathbb C}\Phi_n\cdot h\circ \Phi_n$ is the inverse of (\ref{2.6}), the mapping in (\ref{2.6}) is onto which completes the proof.
\end{proof}
By Lemma \ref{lem2.1} and the fact that $|P_{\mathbb G^n}(f)(z)|\leq P^+_{\mathbb G^n}(|f|)(z)$, the next two theorems are sufficient to yield Theorem \ref{main}.
\begin{thm}\label{main10}  $P_{\mathbb D^n}$ and $P^+_{\mathbb D^n}$ are bounded on $L^p(\mathbb D^n,|J_{\mathbb C}\Phi_n|^{2-p})$ for $p\in \left(\frac{2n}{n+1}, \frac{2n}{n-1}\right)$.
\end{thm}
Theorem 2.3 appears as \cite[Remark 1.5]{CJY23} with the same range of $p$.
\begin{thm}\label{main1} 
$P_{\mathbb D^n}$ is unbounded on $L^p_{\text{anti}}(\mathbb D^n,|J_{\mathbb C}\Phi_n|^{2-p})$ for $p=\frac{2n}{n-1}$.
\end{thm}

Last, we reference below the Forelli-Rudin estimates on $\mathbb D$  which will be used in the proof of Theorem \ref{main1}. See for example \cite{Zhu} for its proof.
\begin{lem}[Forelli-Rudin]\label{lemma2.1} 
	For $\epsilon<1$ and $z\in\mathbb D$, let 
	\begin{equation}\label{**}
		a_{\epsilon,s}(z)=\int_{\mathbb D}\frac{(1-|w|^2)^{-\epsilon}}{|1-z\bar w |^{2-\epsilon-s}}dV(w),
	\end{equation}
	Then \begin{enumerate}
		\item for $s>0$, $a_{\epsilon,s}(z)$ is bounded on $\mathbb D$;
		\item for $s=0$,  $a_{\epsilon,s}(z)$ is comparable to the function $-\log(1-|z|^2)$;
		\item for $s<0$, $a_{\epsilon,s}(z)$ is comparable to the function $(1-|z|^2)^{s}$.
	\end{enumerate}
\end{lem}

\section{Proof of Theorem \ref{main10}} While \cite[Remark 1.5]{CJY23} sketches the proof of Theorem 2.3, here we provide all the relevant details to make the paper self-contained. See also \cite[Corollary 6.1]{YZ}. The main ingredient of the weighted norm estimates of the positive Bergman operator $P^+_{\mathbb D}$ over weighted $L^p$ spaces. 
On the unit disc $\mathbb D$, the boundedness of $P_{\mathbb D}$ and $P^+_{\mathbb D}$ on weighted $L^p$ spaces is closely related to the Bekoll\'e-Bonami constant of weight functions.
Let $T_z$ denote the Carleson tent over $z$ in the unit disc $\mathbb D$ defined as below:
 \begin{itemize}
	\item $T_z:=\left\{w\in \mathbb D:\left|1-\bar w\frac{z}{|z|}\right|<1-|z|\right\}$ for $z\neq 0$, and
	\item $T_z:= \mathbb D$ for $z=0$.
\end{itemize} Bekoll\'e and Bonami \cite{BB78} characterized weighted $L^p$ spaces where $P_{\mathbb D}$ and $P^+_{\mathbb D}$  are bounded:
\begin{thm}[Bekoll\'e-Bonami {\cite{BB78}}]
	Let the weight $u(w)$ be a positive, locally integrable function on the unit disc $\mathbb D$. Let $1<p<\infty$. Then the following conditions are equivalent:
	\begin{enumerate}
		\item $P:L^p(\mathbb D,u)\mapsto L^p(\mathbb D,u)$ is bounded.
		\item $P^+:L^p(\mathbb D,u)\mapsto L^p(\mathbb D,u)$ is bounded.
		\item The Bekoll\'e-Bonami constant $$B_p(u):=\sup_{z\in \mathbb D}\frac{\int_{T_z}u(w) dV(w)}{\int_{T_z}dV(w)}\left(\frac{\int_{T_z}u^{-\frac{1}{p-1}} (w)dV(w)}{\int_{T_z}dV(w)}\right)^{p-1}$$ is finite.
	\end{enumerate} 
\end{thm}
Using dyadic harmonic analysis technique, various authors established quantitative weighted $L^p$ norm estimates of the Bergman projection. See \cite{Pott,Rahm,ZhenghuiWick2,HWW2}. 
\begin{thm}[{\hspace{1sp}\cite[Lemma 15]{Rahm}}]\label{thm2.4}
Let the weight function $u$ be positive, locally integrable on $\mathbb D$. Then for $p\in (1,\infty)$,
\[\|P_{\mathbb D}\|_{L^p(\mathbb D,u)}\leq \|P^+_{\mathbb D}\|_{L^p(\mathbb D,u)}\lesssim \left(B_p(u)\right)^{\max\{1,(p-1)^{-1}\}}.\] 
\end{thm}
\begin{lem}\label{lem2.5}
For a fixed point $a\in \mathbb D$, let  $u_p(w)=|a-w|^{2-p}$. Then
 for any $p\in (4/3,4)$,  $B_p(u_p)\lesssim 1$ where the upper bound is independent of $a$.
Moreover, if we choose arbitrary $m$ points $a_1,\dots,a_m$ in  $\mathbb D$, and set
$$v_p(w)=\prod_{j=1}^{m}|a_j-w|^{2-p},$$ then   for any $p\in (\frac{2m+2}{m+2},\frac{2m+2}{m})$,  $B_p(v_p)\lesssim 1$.  Here the upper bounds may depend on constants $m$ and $p$ but are independent of $a_j$.
\end{lem}
\begin{proof}
We first consider the case of the weight $u_p$. Note that $u_p$ and $u_p^{-1/(p-1)}$ are integrable on $\mathbb D$ if and only if $p\in (\frac{4}{3},4)$. Then it enough to show that $B_p(|a-w|^{b})\lesssim 1$ with an upper bound independent of $a$ if both $u_p$ and $u_p^{-1/(p-1)}$ are integrable on $\mathbb D$. We consider the integral of $u_p$ and $u_p^{-1/(p-1)}$ over $T_z$ for arbitrary $z\in \mathbb D$. 
Notice that  $T_z=\mathbb D\cap\{w:|w-\frac{z}{|z|}|<1-|z|\}$ is the intersection set of the unit disc $\mathbb D$ and the disc
centered at the point $z/|z|$ with Euclidean radius $1 -|z|$. A geometric consideration then
yields that the Lebesgue measure $V(T_z)$ of $T_z$ is comparable to $(1-|z|)^2$. 

If $|a-z|<3(1-|z|)$, then $T_z$ is contained in a ball $B_a$ given by
\[B_a=\{w\in \mathbb C: |w-a|<5(1-|z|)\}.\] Thus,
\begin{align*}
&\frac{\int_{T_z}u_p(w) dV(w)}{\int_{T_z}dV(w)}\left(\frac{\int_{T_z}u_p^{-\frac{1}{p-1}} (w)dV(w)}{\int_{T_z}dV(w)}\right)^{p-1}\\\lesssim& \frac{\int_{B_a}|w-a|^{2-p} dV(w)\left(\int_{B_a}|w-a|^{(p-2)/(p-1)} dV(w)\right)^{p-1}}{(1-|z|)^{2p}}\\=&\frac{(5(1-|z|))^{4-p}\cdot \left((p-1)(3p-4)^{-1}(5(1-|z|))^{(3p-4)/(p-1)}\right)^{p-1}}{(4-p)5^{2p}(1-|z|)^{2p}}\\=&\frac{(p-1)^{p-1}}{(4-p)(3p-4)^{p-1}},
	\end{align*}
provided $u_p$ and $u_p^{-1/(p-1)}$ are integrable.
If $|a-z|\geq 3(1-|z|)$, then $ |a-w|\approx |a-z|$ for all $w\in T_z$ and hence
\begin{align*}
	&\frac{\int_{T_z}u_p(w) dV(w)}{\int_{T_z}dV(w)}\left(\frac{\int_{T_z}u_p^{-\frac{1}{p-1}} (w)dV(w)}{\int_{T_z}dV(w)}\right)^{p-1}\\\lesssim& 	\frac{|a-z|^{2-p}\int_{T_z} dV(w)}{\int_{T_z}dV(w)}\left(\frac{|a-z|^{(p-2)/(p-1)}\int_{T_z} dV(w)}{\int_{T_z}dV(w)}\right)^{p-1}\\=&1.
\end{align*}
Since the upper bound obtained in both cases are independent of the choice of $a$ and $T_z$, we conclude that $B_p(u_p)$ is bounded above by a constant if and only if $p\in (4/3,4)$ and the upper bound is independent of $a_j$.

Now we turn to the case of weight $v_p(w)=\prod_{j=1}^{m}|a_j-w|^{2-p}$. By a similar proof as above, $B_p(|a-w|^{(2-p)m})\lesssim1$ for any $p\in(\frac{2m+2}{m+2},\frac{2m+2}{m})$ where the upper bound is independent of $a$. Using H\"older's inequality, we obtain for any $z\in \mathbb D$
\begin{align*}
	&\frac{\int_{T_z}v_p(w) dV(w)}{\int_{T_z}dV(w)}\left(\frac{\int_{T_z}v_p^{-\frac{1}{p-1}} (w)dV(w)}{\int_{T_z}dV(w)}\right)^{p-1}\\\lesssim& 	\left(\prod_{j=1}^{m}\left(\frac{\int_{T_z} |a_j-w|^{m(2-p)}dV(w)}{\int_{T_z}dV(w)}\right)^{\frac{1}{mp}}\left(\frac{\int_{T_z} |a_j-w|^{m(p-2)/(p-1)}dV(w)}{\int_{T_z}dV(w)}\right)^{\frac{1}{m}-\frac{1}{mp}}\right)^p\\\lesssim&\left(\prod_{j=1}^{m}B_p\left(|a_j-w|^{m(2-p)}\right)\right)^{\frac{p}{m}}\lesssim 1.
\end{align*}
Therefore, $B_p(v_p)\lesssim 1$ with upper bound independent of points $a_j$.
\end{proof}
With Lemma \ref{lem2.5}, we are ready to show Theorem \ref{main10}:
\begin{proof}[Proof of Theorem \ref{main10}]
Since $|P_{\mathbb D^n}(h)(z)|\leq P^+_{\mathbb D^n}(|h|)(z)$ for any $h\in L^p(\mathbb D^n,|J_{\mathbb C}\Phi_n|^{2-p})$, it suffices to show the boundedness for $P^+_{\mathbb D^n}$.
Note that $J_{\mathbb C}\Phi_n(w)$ consists of $n-1$ many factors of each variable $w_j$. When integrating with respect to the single variable $w_j$, only these  $n-1$  factors matter in $J_{\mathbb C}\Phi_n(w)$. Thus the boundedness of $P^+_{\mathbb D^n}$ on $L^p(\mathbb D^n,|J_{\mathbb C}\Phi_n|^{2-p})$ for $p\in\left(\frac{2n}{n+1}, \frac{2n}{n-1}\right)$ follows from  Fubini and Lemma \ref{lem2.5} with $m=n-1$.\end{proof}

\section{Proof of Theorem \ref{main1}}
We will first prove the theorem for the case $n=2$, clearly illustrating the decomposition we use for $P_{\mathbb D^2}$. Then we dive into the case for general $n$ where the decomposition procedure and estimations are more complicated yet the same  strategy applies. 
\subsection{The case for $n=2$} Note that $J_{\mathbb C}\Phi_2=w_1-w_2$. To prove Theorem \ref{main1}, we consider the decomposition $P_{\mathbb D^2}=T^2_1+T^2_2$ where
\begin{align}
	T^2_1(f)(z_1,z_2)&=\int_{\mathbb D^2}\frac{f(w_1,w_2)dV}{\pi^2(1-z_1\bar w_1)(1-z_2\bar w_2)(1-z_1\bar w_2)(1-z_2\bar w_1)}, \\
	T^2_2(f)(z_1,z_2)&=\int_{\mathbb D^2}\frac{(z_1-z_2)(\bar w_1-\bar w_2)f(w_1,w_2)dV}{\pi^2(1-z_1\bar w_1)^2(1-z_2\bar w_2)^2(1-z_1\bar w_2)(1-z_2\bar w_1)}.
\end{align}
\begin{lem}\label{lem2.3}
	$T^2_1$ is a zero operator on $L^p_{\text{anti}}(\mathbb D^2,|w_1-w_2|^{2-p})$.
\end{lem}
\begin{proof}
	Note that $T^2_1(f)(z_1,z_2)$ is symmetric by its definition. For any $f\in L^p_{\text{anti}}(\mathbb D^2,|w_1-w_2|^{2-p})$,  $$T^2_1(f)(z_1,z_2)=T^2_1(-f)(z_2,z_1)=-T^2_1(f)(z_1,z_2),$$which implies $T^2_1(f)= 0$. 
\end{proof}
By  Lemma \ref{lem2.3}, $P_{\mathbb D^2}=T^2_2$  on $L^p_{\text{anti}}(\mathbb D^2,|w_1-w_2|^{2-p})$. So, Theorem \ref{main1} can be further reduced into the following statement in the case $n=2$.
\begin{thm}\label{main2}
	$T^2_2$ is unbounded on  $L^p_{\text{anti}}(\mathbb D^2,|w_1-w_2|^{2-p})$ for $p=4=\frac{2\times 2}{2-1}$.
\end{thm}
\begin{proof}	Let $\tilde T^2$  denote the operator given as follows:
	\[\tilde T^2(h)(z):=(J_{\mathbb C}\Phi_2(z))^{-1}T_2(h\bar J_{\mathbb C}\Phi_2)(z).\]
	
	Then \begin{align}
		\tilde T^2(h)(z)&=\int_{\mathbb D^2}\frac{(\bar w_1-\bar w_2)^2h(w)dV}{\pi^2(1-z_1\bar w_1)^2(1-z_2\bar w_2)^2(1-z_1\bar w_2)(1-z_2\bar w_1)},
	\end{align} and $\|T^2_2\|_{L^p_{\text{anti}}(\mathbb D^2,|J_{\mathbb C}\Phi_2|^{2-p})}=\|\tilde T^2\|_{L^p_{\text{sym}}(\mathbb D^2,|J_{\mathbb C}\Phi_2|^{2})}$ provided one of the norms is finite. Thus it suffices to show that $\tilde T^2$ is unbounded on $L^p_{\text{sym}}(\mathbb D^2,|J_{\mathbb C}\Phi_2|^{2})$ for $p=4$.
For $s\in[\frac{1}{2},1)$,
we set $$h_s(w)=\frac{1}{\pi(1-sw_1)^2}+\frac{1}{\pi(1-sw_2)^2}.$$
	Then
	\begin{align*}
		\|h_s\|^4_{L^4_{\text{sym}}(\mathbb D^2,|J_{\mathbb C}\Phi_2|^{2})}&=\int_{\mathbb D^2}\left|\frac{1}{\pi(1-sw_1)^2}+\frac{1}{\pi(1-sw_2)^2}\right|^4| w_1- w_2|^2dV(w)\\&\lesssim\int_{\mathbb D}\frac{1}{\pi^{4}|1-sw_1|^{8}}\int_{\mathbb D}| w_1- w_2|^{2}dV(w_2)dV(w_1)
		\\&\approx (1-s)^{-6},
	\end{align*}
	where the last equality follows from the Forelli-Rudin estimates (\ref{**}).
	Note that the kernel function of $\tilde{T}^2$ is anti-holomorphic in $w$ variables and $h_s$ can be expressed in terms the conjugate of the Bergman kernels:
	\begin{align*}
	\sum_{j=1}^2\frac{1}{\pi(1-sw_j)^2}={\pi}\left(\overline{K_{\mathbb D^2}((s,0);(\bar w_1,0))}+\overline{K_{\mathbb D^2}((0,s);(0,\bar w_2))}\right).
	\end{align*}
The reproducing property of the Bergman projection implies:
	\begin{align*}
		\tilde T^2(h_s)(z)&=\int_{\mathbb D^2}\frac{(\bar w_1-\bar w_2)^2}{\pi^2(1-z_1\bar w_1)^2(1-z_2\bar w_2)^2(1-z_1\bar w_2)(1-z_2\bar w_1)}\sum_{j=1}^2\frac{1}{\pi(1-sw_j)^2}dV(w)\\&=\frac{s^2}{\pi(1-z_1s)^2(1-z_2s)}+\frac{s^2}{\pi(1-z_2s)^2(1-z_1s)}.
	\end{align*}
	Thus
	\begin{align}\label{4.181}
		\|\tilde T^2(h_s)\|^4_{L^4_{\text{sym}}(\mathbb D^2,|J_{\mathbb C}\Phi_2|^{2})}&=\int_{\mathbb D^2}\left|\frac{s^2}{\pi(1-z_1s)^2(1-z_2s)}+\frac{s^2}{\pi(1-z_2s)^2(1-z_1s)}\right|^4|z_1-z_2|^2dV(z)\nonumber\\&=\int_{\mathbb D^2}\left|\frac{1}{1-z_1s}+\frac{1}{1-z_2s}\right|^4\frac{s^8|z_1-z_2|^2}{\pi^4|1-z_1s|^4|1-z_2s|^4}dV(z).
	\end{align}
	For fixed $s<1$, set $U(s)=\{z\in\mathbb D: \text{Arg}(1-zs)\in \left(-\frac{\pi}{6},\frac{\pi}{6}\right)\}$. Then
	for $z_1,z_2\in U(s)$, 
	\[\left|\frac{1}{1-z_1s}+\frac{1}{1-z_2s}\right|\geq \frac{1}{2|1-z_1s|}.\]
	Applying this inequality to (\ref{4.181}) gives 
	\begin{align*}
		&\int_{\mathbb D^2}\left|\frac{1}{1-z_1s}+\frac{1}{1-z_2s}\right|^4\frac{s^8|z_1-z_2|^2}{\pi^4|1-z_1s|^4|1-z_2s|^4}dV(z)\gtrsim\int_{U^2(s)}\frac{|z_1-z_2|^2}{|1-z_1s|^8|1-z_2s|^4}dV(z).\end{align*}Since \[\frac{z_1-z_2}{(1-z_1s)(1-z_2s)}=\frac{1}{s(1-z_1s)}-\frac{1}{s(1-z_2s)},\]we have
	\begin{align*}&\int_{U^2(s)}\frac{|z_1-z_2|^2}{|1-z_1s|^8|1-z_2s|^4}dV(z)\\=&\int_{U^2(s)}\frac{1}{|1-z_1s|^6|1-z_2s|^2}\left|\frac{1}{s(1-z_1s)}-\frac{1}{s(1-z_2s)}\right|^2dV(z)
		\\=&\int_{U^2(s)}\frac{1}{s^2|1-z_1s|^6|1-z_2s|^2}\left(\frac{1}{|1-z_1s|^2}+\frac{1}{|1-z_2s|^2}-2\text{Re} \frac{1}{(1-z_1s)(1-\bar z_2 s)}\right)dV(z)\\\geq&\int_{U^2(s)}\frac{1}{|1-z_1s|^8|1-z_2s|^2}+\frac{1}{|1-z_1s|^6|1-z_2s|^4}-2 \frac{1}{|1-z_1s|^7|1-z_2s|^3}dV(z).
	\end{align*}
	By realizing that $|1-zs|=s|\frac{1}{s}-z|$ and applying polar coordinates, one can obtain the following Forelli-Rudin estimates (\ref{**}) on $U(s)$. 
	\[\int_{U(s)}\frac{1}{|1-zs|^a}dV(z)\approx\begin{cases}
		(1-s)^{2-a}& a>2\\-\log(1-s)& a=2\\1&a<2.
	\end{cases}\]
	We leave the details of its proof to readers as an exercise.
	Using these estimates,
	\begin{align*}
		&\int_{U^2(s)}\frac{1}{|1-z_1s|^8|1-z_2s|^2}dV(z)\approx -(1-s)^{-6}\log(1-s)\\&\int_{U^2(s)}\frac{1}{|1-z_1s|^6|1-z_2s|^4}dV(z)\approx \int_{U^2(s)} \frac{1}{|1-z_1s|^7|1-z_2s|^3}dV(z)\approx (1-s)^{-6},
	\end{align*}
	which implies that	$\|\tilde T^2(h_s)\|^4_{L^4_{\text{sym}}(\mathbb D^2,|J_{\mathbb C}\Phi_2|^{2})}\approx -(1-s)^{-6}\log(1-s)$. 
	
	As $s\to 1$,
	\[\frac{\|\tilde T^2(h_s)\|^4_{L^4_{\text{sym}}(\mathbb D^2,|J_{\mathbb C}\Phi_2|^{2})}}{\| h_s\|^4_{L^4_{\text{sym}}(\mathbb D^2,|J_{\mathbb C}\Phi_2|^{2})}}\gtrsim-\log(1-s)\to \infty,\]
	proving that  $\tilde T^2$ is unbounded on $L^4_{\text{sym}}(\mathbb D^2,|J_{\mathbb C}\Phi_2|^{2})$. \end{proof}
\subsection{The case for general $n$} Like the case $n=2$, our proof for general $n$ also involves a  decomposition of $P_{\mathbb D^n}$ into operators $T^n_1$ and $T^n_2$. 
\begin{align}\label{1.3}T^n_{1}(h)(z)&=\int_{\mathbb D^n}\frac{\prod_{1\leq j<k\leq n}(1-z_k\bar w_j)(1-z_j\bar w_k)-\prod_{1\leq j<k\leq n}(z_j-z_k)(\bar w_j-\bar w_k)}{\pi^n\prod_{1\leq j\leq k\leq n}(1-z_k\bar w_j)(1-z_k\bar w_j)}h(w)dV(w)\\\label{1.4}
T^n_{2}(h)(z)&=(P_{\mathbb D^n}-T^n_{1})(h)(z)=\int_{\mathbb D^n}\frac{\prod_{1\leq j<k\leq n}(z_j-z_k)(\bar w_j-\bar w_k)}{\pi^n\prod_{1\leq j\leq k\leq n}(1-z_k\bar w_j)(1-z_j\bar w_k)}h(w)dV(w).
\end{align}
\begin{lem}\label{lem4.3}
	$T^n_1$ 
	is a zero operator on $L^p_{\text{anti}}(\mathbb D^n,|J_{\mathbb C}\Phi_n|^{2-p})$.
\end{lem}
\begin{proof}
	Recall that $\tau_{j,k}$ is the permutation that interchanges variables $w_j$ and $w_k$, and a kernel function $K(z;\bar w)$ on $\mathbb D^n\times \mathbb D^n$ is called $(j,k)$-symmetric in $w$ if  $K(z;\bar w)=K(z;\bar \tau_{j,k}(w))$.  
If $K(z;\bar w)$  is $(j,k)$-symmetric in $\bar w$, then for any anti-symmetric $f\in L^p_{\text{anti}}(\mathbb D^n,|J_{\mathbb C}\Phi_n|^{2-p})$, we have
\begin{align*}\int_{\mathbb D^n}K(z;\bar w)f(w)dV(w)&=-\int_{\mathbb D^n}K(z;\bar\tau_{j,k} (w))f(\tau_{j,k}(w))dV(w)\\&=-\int_{\mathbb D^n}K(z;\bar w)f(w)dV(w).\end{align*}
Thus operators with $(j,k)$-symmetric kernel functions in $w$ annihilate  $ L^p_{\text{anti}}(\mathbb D^n,|J_{\mathbb C}\Phi_n|^{2-p})$.

For $l=1,\dots,n$, we define the operator $P_l$ to be as follows:
\begin{equation}\label{4.16}P_{l}(h)(z)=\int_{\mathbb D^n}\frac{\prod_{1\leq j<k\leq l}(1-z_j\bar w_k)(1-z_k\bar w_j)\prod_{1\leq j<k\leq n,1\leq l< k\leq n}(z_j-z_k)(\bar w_j-\bar w_k)}{\pi^n\prod_{1\leq j\leq k\leq n}(1-z_k\bar w_j)(1-z_j\bar w_k)}h(w)dV(w).\end{equation}
Then $P_1=T^n_2$ and $P_n=P_{\mathbb D^n}$. 
	We claim that $P_{\mathbb D^n}=P_{l}$  on $L^p_{\text{anti}}(\mathbb D^n,|J_{\mathbb C}\Phi_n|^{2-p})$ for all $l=1,\dots,n$. Then $T^n_1=P_{\mathbb D^n}-P_1=0$ on $L^p_{\text{anti}}(\mathbb D^n,|J_{\mathbb C}\Phi_n|^{2-p})$.
 We prove the claim by induction on $l$. 
 
 Let $K_l$ denote the kernel function of $P_l$.
 When $l=2$,
\begin{align*}K_{2}(z;\bar w)&=\frac{(1-z_1\bar w_2)(1-z_2\bar w_1)\prod_{1\leq j<k\leq n,(j,k)\neq(1,2)}(z_j-z_k)(\bar w_j-\bar w_k)}{\pi^n\prod_{1\leq j\leq k\leq n}(1-z_k\bar w_j)(1-z_j\bar w_k)}.\end{align*}
Then
\begin{align*}&K_{2}(z;\bar w)-K_{1}(z;\bar w)\\=&\frac{((1-z_1\bar w_2)(1-z_2\bar w_1)-(z_1-z_2)(\bar w_1-\bar w_2))\prod_{1\leq j<k\leq n,(j,k)\neq(1,2)}(z_j-z_k)(\bar w_j-\bar w_k)}{\pi^n\prod_{1\leq j\leq k\leq n}(1-z_k\bar w_j)(1-z_j\bar w_k)}\\=&\frac{(1-z_1\bar w_1)(1-z_2\bar w_2)\prod_{1\leq j<k\leq n,(j,k)\neq(1,2)}(z_j-z_k)(\bar w_j-\bar w_k)}{\pi^n\prod_{1\leq j\leq k\leq n}(1-z_k\bar w_j)(1-z_j\bar w_k)}\\=&\frac{\prod_{1\leq j<k\leq n,(j,k)\neq(1,2)}(z_j-z_k)(\bar w_j-\bar w_k)}{\pi^n\prod_{j=3}^n(1-z_j\bar w_j)\prod_{j,k=1}^n(1-z_k\bar w_j)}.\end{align*}
It is not hard to check that $K_{2}-K_{1}$ is $(1,2)$-symmetric in $w$ which shows that $P_{1}=P_2$ on $L^p_{\text{anti}}(\mathbb D^n,|J_{\mathbb C}\Phi_n|^{2-p})$.

Suppose that $P_{1}=P_l$ on $L^p_{\text{anti}}(\mathbb D^n,|J_{\mathbb C}\Phi_n|^{2-p})$ for $l=m$. We show that $P_{m+1}=P_m$ on $L^p_{\text{anti}}(\mathbb D^n,|J_{\mathbb C}\Phi_n|^{2-p})$. Let $\mathcal R_m$ denote the power set \[\mathcal R_m:=\{\mathcal I:\mathcal I\subseteq \{1,2,\dots,m\}\}.\] Given $\mathcal I\in\mathcal R_m$, let $|\mathcal I|$ denote the cardinality of $\mathcal I$. 
For simplicity of notation, we set
$a_{j,k}=1-z_j\bar w_k$
and $b_{j,k}=(z_j-z_k)(\bar w_j-\bar w_k)$. Then for $j\neq k$, $a_{j,k}a_{k,j}=a_{j,j}a_{kk}+b_{j,k}.$
Note that
\begin{align*}
	\prod_{j=1}^{m}a_{j,m+1} a_{m+1,j}=&		\prod_{j=1}^{m}(a_{j,j} a_{m+1,m+1}+b_{j,m+1})\\=&	\sum_{\mathcal I\in \mathcal R_m}a_{m+1,m+1}^{|\mathcal I|}\prod_{j\in\mathcal I}a_{j,j} \prod_{k\in \mathcal I^c}b_{k,m+1}.
\end{align*}
We set
\[p_{\mathcal I}(z;\bar w):=a_{m+1,m+1}^{|\mathcal I|}\prod_{j\in\mathcal I}a_{j,j} \prod_{k\in \mathcal I^c}b_{k,m+1}.\] 
Then $$\prod_{j=1}^{m}a_{j,m+1}a_{m+1,j}=\prod_{j=1}^{m}(1-z_j\bar w_{m+1})(1-z_{m+1}\bar w_j)=\sum_{\mathcal I\in \mathcal R_m}p_{\mathcal I}(z;\bar w).$$
Let $K_m$ and $K_{m+1}$ be the kernel function of $P_m$ and $P_{m+1}$ respectively as in (\ref{4.16}). Let $K_{m,\mathcal I}$ denote the kernel function 
\[K_{m,\mathcal I}(z;\bar w):=\frac{p_{\mathcal I}(z;\bar w)\prod_{j<k\leq m}a_{j,k}a_{k,j}\prod_{j<k,m+1<k}b_{j,k}}{\pi^n\prod_{j\leq k}a_{j,k}a_{k,j}}.\]
We can express $K_m$ and $K_{m+1}$ in terms of $K_{m,\mathcal I}(z;\bar w)$:
\begin{align*}
	K_m(z;\bar w)&=\frac{\prod_{j<k\leq m}(1-z_j\bar w_k)(1-z_k\bar w_j)\prod_{j<k, m<k}(z_j-z_k)(\bar w_j-\bar w_k)}{\pi^n\prod_{j\leq k}(1-z_k\bar w_j)(1-z_j\bar w_k)}\\&=\frac{p_{\emptyset}(z;\bar w)\prod_{j<k\leq m}a_{j,k}a_{k,j}\prod_{j<k, m+1<k}b_{j,k}}{\pi^n\prod_{j\leq k}a_{j,k}a_{k,j}}\\&=K_{m,\emptyset}(z;\bar w),
\end{align*}
and 
\begin{align*}
	K_{m+1}(z;\bar w)&=\frac{\prod_{j<k\leq m+1}(1-z_j\bar w_k)(1-z_k\bar w_j)\prod_{j<k,m+1<k}(z_j-z_k)(\bar w_j-\bar w_k)}{\pi^n\prod_{j\leq k}(1-z_k\bar w_j)(1-z_j\bar w_k)}\\&=\frac{\sum_{\mathcal I\in \mathcal R_m}p_{\mathcal I}(z;\bar w)\prod_{j<k\leq m}(1-z_j\bar w_k)(1-z_k\bar w_j)\prod_{j<k,m+1<k}(z_j-z_k)(\bar w_j-\bar w_k)}{\pi^n\prod_{j\leq k}(1-z_k\bar w_j)(1-z_j\bar w_k)}\\&=\sum_{\mathcal I\in \mathcal R_m}K_{m,\mathcal I}(z;\bar w)=K_{m}(z;\bar w)+\sum_{\emptyset\neq\mathcal I\in \mathcal R_m}K_{m,\mathcal I}(z;\bar w).
\end{align*}
We show that for any nonempty $\mathcal I\in\mathcal R_m$, $K_{m,\mathcal I}$ is a linear combination of $(j,k)$-symmetric kernel functions. Then for anti-symmetric $f\in L^p_{\text{anti}}(\mathbb D^n,|J_{\mathbb C}\Phi_n|^{2-p})$,
\begin{align*}P_{m+1}(f)(z)&=\int_{\mathbb D^n}K_{m+1}(z;\bar w)f(w)dV(w)\\&=\int_{\mathbb D^n}\sum_{\mathcal I\in\mathcal R_m}K_{m,\mathcal I}(z;\bar w)f(w)dV(w)\\&=\int_{\mathbb D^n}K_{m}(z;\bar w)f(w)dV(w)\\&=P_{m}(f)(z),
\end{align*} which completes the induction and the proof of the lemma.
When $|\mathcal I|>1$, there exists $j_1,j_2\in \mathcal I$, and 
\begin{align*}K_{m,\mathcal I}(z;\bar w)&=\frac{p_{\mathcal I}(z;\bar w)\prod_{j<k\leq m}a_{j,k}a_{k,j}\prod_{j<k,m+1<k}b_{j,k}}{\pi^n\prod_{j\leq k}a_{j,k}a_{k,j}}\\&=\frac{a_{m+1,m+1}^{|\mathcal I|}\prod_{k\in\mathcal I}a_{k,k} \prod_{j\in \mathcal I^c}b_{j,m+1}\prod_{ j<k\leq m}a_{j,k}a_{k,j}\prod_{j<k,m+1<k}b_{j,k}}{\pi^n\prod_{j\leq k}a_{j,k}a_{k,j}}.\end{align*}
It's easy to see that $K_{m,\mathcal I}(z;\bar w)$ is $(j_1,j_2)$-symmetric. 

Now we turn to consider the case when $\mathcal I=\{j_0\}$. Without loss of generality, we let $j_0=1$. 
\begin{align*}K_{m,\{1\}}(z;\bar w)=&\frac{p_{\{1\}}(z;\bar w)\prod_{j<k\leq m}a_{j,k}a_{k,j}\prod_{j<k,m+1<k}b_{j,k}}{\pi^n\prod_{j\leq k}a_{j,k}a_{k,j}}\\=&\frac{a_{m+1,m+1}a_{1,1} \prod_{k=2}^mb_{j,m+1}\prod_{j<k\leq m}a_{j,k}a_{k,j}\prod_{j<k,m+1<k}b_{j,k}}{\pi^n\prod_{j\leq k}a_{j,k}a_{k,j}}\\=&\frac{a_{m+1,m+1}a_{1,1} (a_{2,m+1}a_{m+1,2}-a_{2,2}a_{m+1,m+1})\prod_{k=3}^mb_{j,m+1}\prod_{j<k\leq m}a_{j,k}a_{k,j}\prod_{j<k,m+1<k}b_{j,k}}{\pi^n\prod_{j\leq k}a_{j,k}a_{k,j}}
	\\=&\frac{a_{m+1,m+1}a_{1,1} a_{2,m+1}a_{m+1,2}\prod_{k=3}^mb_{j,m+1}\prod_{j<k\leq m}a_{j,k}a_{k,j}\prod_{j<k,m+1<k}b_{j,k}}{\pi^n\prod_{j\leq k}a_{j,k}a_{k,j}}-K_{m,\{1,2\}}(z;\bar w),
\end{align*}
where $K_{m,\{1,2\}}(z;\bar w)$ is $(1,2)$-symmetric in $w$.

Since $b_{3,m+1}=a_{3,m+1}a_{m+1,3}-a_{3,3}a_{m+1,m+1}$, we have
\begin{align*}&\frac{a_{m+1,m+1}a_{1,1} a_{2,m+1}a_{m+1,2}\prod_{k=3}^mb_{j,m+1}\prod_{j<k\leq m}a_{j,k}a_{k,j}\prod_{j<k,m+1<k}b_{j,k}}{\pi^n\prod_{j\leq k}a_{j,k}a_{k,j}}\\=&\frac{a_{m+1,m+1}a_{1,1} a_{2,m+1}a_{m+1,2}a_{3,m+1}a_{m+1,3}\prod_{k=4}^mb_{j,m+1}\prod_{j<k\leq m}a_{j,k}a_{k,j}\prod_{j<k,m+1<k}b_{j,k}}{\pi^n\prod_{j\leq k}a_{j,k}a_{k,j}}\\&-\frac{a^2_{m+1,m+1}a_{1,1} a_{2,m+1}a_{m+1,2}a_{3,3}\prod_{k=4}^mb_{j,m+1}\prod_{j<k\leq m}a_{j,k}a_{k,j}\prod_{j<k,m+1<k}b_{j,k}}{\pi^n\prod_{j\leq k}a_{j,k}a_{k,j}},
\end{align*}
where the negative term above is $(1,3)$-symmetric in $w$.
Repeating the above process using the identity $b_{j,m+1}=a_{j,m+1}a_{m+1,j}-a_{j,j}a_{m+1,m+1}$ until no $b_{j,m+1}$ term left, we obtain 
\begin{align*}K_{m,\{1\}}(z;\bar w)-\frac{a_{m+1,m+1}a_{1,1}\prod_{k=2}^m a_{k,m+1}a_{m+1,k}\prod_{j<k\leq m}a_{j,k}a_{k,j}\prod_{j<k,m+1<k}b_{j,k}}{\pi^n\prod_{j\leq k}a_{j,k}a_{k,j}}
\end{align*} 
is a linear combination of functions  that are $(1,j)$-symmetric in $w$.
Since the function 
\[\frac{a_{m+1,m+1}a_{1,1}\prod_{k=2}^m a_{k,m+1}a_{m+1,k}\prod_{j<k\leq m}a_{j,k}a_{k,j}\prod_{j<k,m+1<k}b_{j,k}}{\pi^n\prod_{j\leq k}a_{j,k}a_{k,j}}\]
is $(1,m+1)$-symmetric in $w$, we are done.
\end{proof}
Since $T^n_2=P_{\mathbb D^n}$ on $L^p_{\text{anti}}(\mathbb D^n,|J_{\mathbb C}\Phi_n|^{2-p}|)$, the next theorem implies Theorem \ref{main1} for general $n$.
\begin{thm}\label{thm4.4}
	$T^n_2$ 
	is unbounded on $L^p_{\text{anti}}(\mathbb D^n,|J_{\mathbb C}\Phi_n|^{2-p})$ for $p=\frac{2n}{n-1}$.
\end{thm}
\begin{proof}The proof for the case $n>2$ follows from a similar argument as in the proof of Theorem \ref{main2}. 
	Let $\tilde T^n$  denote the operator given as follows:
\[\tilde T^n(h)(z):=(J_{\mathbb C}\Phi_n(z))^{-1}T^n_2(h\bar J_{\mathbb C}\Phi_n)(z).\]

Then \begin{align}
	\tilde T^n(h)(z)&=\int_{\mathbb D^n}\frac{\prod_{j<k}(\bar w_j-\bar w_k)^2h(w)dV}{\pi^n\prod_{j\leq k}(1-z_k\bar w_j)(1-z_j\bar w_k)},
\end{align} and $\|T^n_2\|_{L^p_{\text{anti}}(\mathbb D^n,|J_{\mathbb C}\Phi_n|^{2-p})}=\|\tilde T^n\|_{L^p_{\text{sym}}(\mathbb D^n,|J_{\mathbb C}\Phi_n|^{2})}$ provided one of the norms is finite. Thus it suffices to show that $\tilde T^n$ is unbounded on $L^p_{\text{sym}}(\mathbb D^n,|J_{\mathbb C}\Phi_n|^{2})$ for $p=\frac{2n}{n-1}$.

 Recall that $\mathcal S_n$ is the set of all permutations of $\{z_1,\dots,z_n\}$. For $s\in(0,1)$, we set
$$h_s(z)=\sum_{\tau\in \mathcal S_n}\frac{1}{\prod_{j=1}^{n-1}(1-\tau(z_j)s)^n}.$$
Then $h_s$ is a symmetric function with
\begin{align}\label{4.190}
	\|h_s\|^p_{L^p_{\text{sym}}(\mathbb D^n,|J_{\mathbb C}\Phi_n|^{2})}\nonumber&=\int_{\mathbb D^n}\left|\sum_{\tau\in \mathcal S_n}\frac{1}{\pi^{n-1}\prod_{l=1}^{n-1}(1-\tau(w_l)s)^n}\right|^p\prod_{1\leq j<k\leq n }| w_j- w_k|^2dV(w)\nonumber\\&\lesssim\int_{\mathbb D^n}\frac{\prod_{1\leq j<k\leq n}| w_j- w_k|^2}{\prod_{l=1}^{n-1}|1-w_ls|^{np}}dV(w)\nonumber\\&\lesssim\int_{\mathbb D^{n-1}}\frac{\prod_{1\leq j<k\leq n-1}| w_j- w_k|^2}{\prod_{l=1}^{n-1}|1-w_ls|^{np}}dV(w_1,\dots,w_{n-1})\nonumber\\&\lesssim\int_{\mathbb D^{n-1}}\frac{\prod_{1\leq j<k\leq n-1}| w_j- w_k|^2}{\prod_{l=1}^{n-1}|1-w_ls|^{2n-4}}\frac{1}{\prod_{l=1}^{n-1}|1-w_ls|^{np+4-2n}}dV(w_1,\dots,w_{n-1}).
\end{align}

To evaluate the integral above we need an $(n-2)$-step procedure to eliminate the numerator of the integrand, i.e.
we rewrite 
\[\frac{\prod_{1\leq j<k\leq n-1}(w_j- w_k)}{\prod_{l=1}^{n-1}(1-w_ls)^{n-2}}.\]
Step 1. Recall that by partial fractions:
\begin{align*}\frac{1}{\prod_{j=1}^{n-1}(1-w_js)}=\sum_{j=1}^{n-1}\frac{c_j}{(1-w_js)},\end{align*}
where
$c_j=\frac{1}{s^{n-2}\prod_{k=1,k\neq j}^{n-1}({w_j}-{w_k})}$. Then 
\begin{align*}\frac{\prod_{1\leq j<k\leq n-1}(w_j- w_k)}{\prod_{l=1}^{n-1}(1-w_ls)^{n-2}}&=\sum_{j_1=1}^{n-1}\frac{\prod_{1\leq j< k\leq n-1}(w_j- w_k)}{s^{n-2}(1-{w}_{j_1}s)\prod_{l=1}^{n-1}(1-w_ls)^{n-3}\prod_{k=1,k\neq j_1}^{n-1}({w_{j_1}}-{w_k})}.
\end{align*}
Step 2. Now we focus on the $j_1$th term in the sum above 
$$\frac{\prod_{1\leq j< k\leq n-1}(w_j- w_k)}{s^{n-2}(1-{w}_{j_1}s)\prod_{l=1}^{n-1}(1-w_ls)^{n-3}\prod_{k=1,k\neq j_1}^{n-1}({w_{j_1}}-{w_k})}.$$
Applying the partial fractions yields
$$\frac{1}{\prod_{j=1,j\neq j_1}^{n-1}(1-w_js)}=\sum_{j=1,j\neq j_1}^{n}\frac{1}{s^{n-3}(1- w_js)\prod_{k=1,k\neq j_1}^{n-1}(w_{j_1}-w_{k})},$$
and 
\begin{align*}
	&\frac{\prod_{1\leq j< k\leq n-1}(w_j- w_k)}{s^{n-2}(1-{w}_{j_1}s)\prod_{l=1}^{n-1}(1-w_ls)^{n-3}\prod_{{k=1,k\neq j_1}}^{n-1}({w_{j_1}}-{w_k})}\\=&\sum_{j_1=1}^{n-1}\sum_{\substack{j_2=1\\j_2\neq j_1}}^{n-1}\frac{\prod_{1\leq j<k\leq n-1}( w_j-w_k)}{s^{2n-5}(1-{w}_{j_1}s)^2(1-{w}_{j_2}s)\prod_{j=1}^{n-1}(1-w_js)^{n-4}\prod_{{k=1,k\neq j_1}}^{n-1}({w_{j_1}}-{w_k})\prod_{{k=1,k\neq j_1,j_2}}^{n-1}({w_{j_2}}-{w_k})}.
\end{align*}

Step 3. As in Step 2, we turn to the term with sub-indices $(j_1,j_2)$ in the sum above and continue the process by doing partial fractions to
$$\frac{1}{\prod_{j=1,j\nin \{j_1,j_2\}}^{n-1}(1- w_js)}.$$
Repeat this process. Then after $n-2$ steps, 
we obtain
\begin{align}\label{4.202}
\frac{\prod_{1\leq j<k\leq n-1}(w_j- w_k)}{\prod_{l=1}^{n-1}(1-w_ls)^{n-2}}\nonumber&=\sum_{(l_1, l_2,\dots,l_{n-1})\in \mathcal S_{n-1}}\frac{s^{-\frac{1}{2}(n-1)(n-2)}\prod_{1\leq j<k\leq n-1}(w_j- w_k)}{\prod_{1\leq j<k\leq n-1}(w_{l_j}- w_{l_k})\prod_{t=1}^{n-1}(1- w_{l_t}s)^{n-1-t}}\nonumber\\&=\sum_{(l_1, l_2,\dots,l_{n-1})\in \mathcal S_{n-1}}\frac{\text{sgn}((l_1,\dots,l_{n-1}))s^{-\frac{1}{2}(n-1)(n-2)}}{\prod_{t=1}^{n-1}(1- w_{l_t}s)^{n-1-t}}.
\end{align}
Here $\text{sgn}((l_1,\dots,l_{n-1}))$ is the sign of the permutation $(l_1,\dots,l_{n-1})$.

Applying this identity to (\ref{4.190}) and using the triangle inequality, we obtain
\begin{align}
	&\|h_s\|^p_{L^p_{\text{sym}}(\mathbb D^n,|J_{\mathbb C}\Phi_n|^{2})}\nonumber\\\lesssim&\int_{\mathbb D^{n-1}}\frac{\prod_{1\leq j<k\leq n-1}| w_j- w_k|^2}{\prod_{l=1}^{n-1}|1-w_ls|^{2n-4}}\frac{1}{\prod_{l=1}^{n-1}|1-w_ls|^{np+4-2n}}dV(w_1,\dots,w_{n-1})\nonumber\\\lesssim&\sum_{(l_1, l_2,\dots,l_{n-1})\in \mathcal S_{n-1}}\int_{\mathbb D^{n-1}}\frac{s^{-(n-1)(n-2)}}{\prod_{t=1}^{n-1}|1- w_{l_t}s|^{2n-2-2t}}.\frac{1}{\prod_{l=1}^{n-1}|1-w_ls|^{np+4-2n}}dV(w_1,\dots,w_{n-1})\nonumber\\\lesssim&\int_{\mathbb D^{n-1}}\frac{1}{\prod_{l=1}^{n-1}|1-w_ls|^{np+2-2l}}dV(w_1,\dots,w_{n-1}).
\end{align}
For $p=\frac{2n}{n-1}$, $np+2-2l\geq np+2-2(n-1)>2$. Thus the Forelli-Rudin estimates (\ref{**}) imply 
\begin{align}
&\int_{\mathbb D^{n-1}}\frac{1}{\prod_{l=1}^{n-1}|1-w_ls|^{np+2-2l}}dV(w_1,\dots,w_{n-1})\nonumber\\=&\prod_{l=1}^{n-1}\int_{\mathbb D}\frac{1}{|1-w_ls|^{np+2-2l}}dV(w_1,\dots,w_{n-1})\nonumber\\\approx&\prod_{l=1}^{n-1}(1-s)^{-np+2l}=(1-s)^{-n^2-n}.
\end{align}
Hence $	\|h_s\|^p_{L^p_{\text{sym}}(\mathbb D^n,|J_{\mathbb C}\Phi_n|^{2})}\lesssim(1-s)^{-n^2-n}.$

Now we turn to compute $\tilde T^n(h_s)$.  Let $I$ denote the identity operator. 
For the variable $w_j$, let $\mathcal D_{w_j}$ denote the partial differential operator
\[\mathcal D_{w_j}=I+w_j\frac{\partial}{\partial w_j}.\]
For any $k\in \mathbb N$ and holomorphic function $f(w)=\sum_{\alpha\in \mathbb N^n} c_{\alpha}w^{\alpha}$ on $\mathbb D^n$,
\[(\mathcal D_{w_j})^kf(w)=\sum_{\alpha\in \mathbb N^n} c_{\alpha}(\alpha_j+1)^kw^\alpha.\]
For each integer $k>2$, $$\frac{1}{(1-w_js)^k}=\sum_{m=0}^{\infty}(m+1)_{k-1}s^mw_j^m=\sum_{m=0}^{\infty}(m+2)_{k-2}((m+1)s^mw_j^m)$$ where the Pochhammer symbol  $(m+2)_{k-2}=(m+2)\dot(m+3)\dots(m+k-1)$ is a polynomial in $m$ of degree $k-2$. Thus, there exists a polynomial $q_{k-2}$ of degree $k-2$ such that 
$$\frac{1}{(1-w_js)^k}=q_{k-2}(\mathcal D_{w_j})\left(\frac{1}{\pi(1-w_js)^2}\right).$$
For holomorphic functions $f,g$ on $\mathbb D^n$ with $f(w)=\sum_{\alpha}c_{\alpha}w^\alpha$ and $g(w)=\sum_{\alpha}d_{\alpha}w^\alpha$,
\begin{align}
	\int_{\mathbb D^n}f\overline{\prod_{j=1}^{n}q_{k-2}(\mathcal D_{w_j})(g)}dV=&\int_{\mathbb D^n}\left(\sum_{\alpha}c_{\alpha}w^\alpha\right) \left(\sum_{\alpha}d_{\alpha}\prod_{j=1}^{n}q_{k-2}(\alpha_j+1)\bar w^\alpha\right)dV(w)\nonumber\\=&\sum_{\alpha}c_{\alpha}d_{\alpha}\prod_{j=1}^{n}q_{k-2}(\alpha_j+1)\int_{\mathbb D^n}|w|^{2\alpha}dV(w)\nonumber\\=&\int_{\mathbb D^n}\left(\sum_{\alpha}c_{\alpha}\prod_{j=1}^{n}q_{k-2}(\alpha_j+1)w^\alpha\right) \left(\sum_{\alpha}d_{\alpha}\bar w^\alpha\right)dV(w)\nonumber\\=&\int_{\mathbb D^n}\prod_{j=1}^{n}q_{k-2}(\mathcal D_{w_j})(f)(w)\bar g(w)dV(w).
\end{align}
Therefore, we have
\begin{align*}
	&\tilde T^n(h_s)(z)\nonumber\\=&\int_{\mathbb D^n}\frac{\prod_{1\leq j< k\leq n}(\bar w_j-\bar w_k)^2}{\pi^n\prod_{m=1}^n(1-z_m\bar w_m)\prod_{j,k=1}^n(1-z_j\bar w_k)}\sum_{\tau\in \mathcal S_n}\frac{1}{\prod_{j=1}^{n-1}(1-\tau(w_j)s)^n}dV(w)\\=&\int_{\mathbb D^n}\frac{\prod_{1\leq j< k\leq n}(\bar w_j-\bar w_k)^2}{\pi^n\prod_{m=1}^n(1-z_m\bar w_m)\prod_{j,k=1}^n(1-z_j\bar w_k)}\sum_{\tau\in \mathcal S_n}\prod_{j=1}^{n-1} q_{n-2}(\mathcal D_{\tau(w_j)})\left(\frac{1}{\pi(1-\tau(w_j)s)^2}\right)dV(w)
	\\=&\int_{\mathbb D^n}\sum_{\tau\in \mathcal S_n}\prod_{j=1}^{n-1} q_{n-2}(\mathcal D_{\tau(\bar w_j)})\left(\frac{\prod_{1\leq j< k\leq n}(\bar w_j-\bar w_k)^2}{\pi^n\prod_{m=1}^n(1-z_m\bar w_m)\prod_{j,k=1}^n(1-z_j\bar w_k)}\right)\left(\frac{1}{\pi^{n-1}\prod_{j=1}^{n-1}(1-\tau(\bar w_j)s)^2}\right)dV(w)
		\\=&\int_{\mathbb D^n}\sum_{\tau\in \mathcal S_n}\prod_{j=1}^{n-1} q_{n-2}(\mathcal D_{\tau(\bar w_j)})\left(\frac{\prod_{1\leq j< k\leq n}(\bar w_j-\bar w_k)^2}{\pi^n\prod_{m=1}^n(1-z_m\bar w_m)\prod_{j,k=1}^n(1-z_j\bar w_k)}\right)\pi K_{\mathbb D^n}(w;\tau(s,\dots,s,0))dV(w)
	\\=&\sum_{\tau\in \mathcal S_n}\prod_{j=1}^{n-1} q_{n-2}(\mathcal D_{\tau(\bar w_j)})\left(\frac{\prod_{1\leq j< k\leq n}(\bar w_j-\bar w_k)^2}{\pi^{n-1}\prod_{m=1}^n(1-z_m\bar w_m)\prod_{j,k=1}^n(1-z_j\bar w_k)}\right)\bigg|_{\bar w=\tau(s,\dots,s,0)}.
\end{align*}
We claim that there is a constant $c_n$ such that
\begin{align}\label{4.220}&\prod_{j=1}^{n-1} q_{n-2}(\mathcal D_{\tau(\bar w_j)})\left(\frac{\prod_{1\leq j< k\leq n}(\bar w_j-\bar w_k)^2}{\pi^{n-1}\prod_{m=1}^n(1-z_m\bar w_m)\prod_{j,k=1}^n(1-z_j\bar w_k)}\right)\bigg|_{\bar w=\tau(s,\dots,s,0)}\nonumber\\=&\frac{c_ns^{n(n-1)}}{\prod_{m=1}^{n-1}(1-\tau(z_m)s)\prod_{l=1}^n(1-z_ls)^{n-1}}.
\end{align}
By symmetry, it suffices to show (\ref{4.220}) for the case when $\tau$ is the identity map, i.e.
\begin{align}\label{4.24}&\prod_{j=1}^{n-1} q_{n-2}(\mathcal D_{\bar w_j})\left(\frac{\prod_{1\leq j< k\leq n}(\bar w_j-\bar w_k)^2}{\pi^{n-1}\prod_{m=1}^n(1-z_m\bar w_m)\prod_{j,k=1}^n(1-z_j\bar w_k)}\right)\bigg|_{\bar w=(s,\dots,s,0)}\nonumber\\=&\frac{c_ns^{n(n-1)}}{\prod_{m=1}^{n-1}(1-z_ms)\prod_{l=1}^n(1-z_ls)^{n-1}}.\end{align}
Set $\bar \partial_j=\frac{\partial}{\partial \bar w_j}$. For a multi-index $\mathbf l=(l_1,\dots,l_{n})$, set $\bar \partial^{\mathbf l}=\bar \partial_1^{l_1}\dots\bar \partial_n^{l_n}$. Then by the product rule, $\mathcal D_{\bar w_j}^k=\sum_{l=0}^{k}c_{k,l}\bar w_j^l\bar \partial^l_j$. Therefore
$$\prod_{j=1}^{n-1}q_{n-2}(\mathcal D_{\bar w_j})=\prod_{j=1}^{n-1}\left(\sum_{l_j=0}^{n-2}d_{l_j} \bar w_j^{l_j}\bar \partial^{l_j}_j\right)=\sum_{\mathbf l\in\{0,1,\dots, n-2\}^{n-1}}d_{l_1}\cdots d_{l_{n-1}}\bar w^{\mathbf l}\bar \partial^{\mathbf l},$$ for some constants $d_{l_j}$.
Note that for $\mathbf l=(l_1,\dots,l_{n-1})\in \{0,1,\dots, n-2\}^{n-1}$,
\[\bar \partial^{\mathbf l}\left(\frac{\prod_{1\leq j< k\leq n}(\bar w_j-\bar w_k)^2}{\pi^{n-1}\prod_{m=1}^n(1-z_m\bar w_m)\prod_{j,k=1}^n(1-z_j\bar w_k)}\right)\]
can be expressed as a linear combination of terms of the form
\[\bar \partial^{\mathbf m}\left({\prod_{1\leq j< k\leq n}(\bar w_j-\bar w_k)^2}\right)\bar \partial^{\mathbf l-\mathbf m}\left(\frac{1}{\pi^{n-1}\prod_{m=1}^n(1-z_m\bar w_m)\prod_{j,k=1}^n(1-z_j\bar w_k)}\right),\]
where $\mathbf m=(m_1,\dots,m_{n-1})$ with $m_j\leq l_j$ for all $j$ and $\mathbf l-\mathbf m=(l_1-m_1,\dots,l_{n-1}-m_{n-1})$. 

Since $l_j\leq n-2$ for each $j$, the sum $$|\mathbf m|=\sum m_j\leq \sum l_j\leq (n-1)(n-2).$$Thus, the polynomial $\bar \partial^{\mathbf m}\left({\prod_{1\leq j< k\leq n}(\bar w_j-\bar w_k)^2}\right)$ is of total degree  $n(n-1)-|\mathbf m|$ which is at least $n(n-1)-(n-1)(n-2)=2(n-1)$. Note also that for $\bar w=(s,\dots,s,0)$, the factor $(\bar w_j-\bar w_k)\neq 0$ if and only if either $j$ or $k$ equals $n$. It is not hard to see that the polynomial $\prod_{k=1}^{n-1}(\bar w_k-\bar w_n)^2$
 is the only divisor of $\prod_{1\leq j< k\leq n}(\bar w_j-\bar w_k)^2$ that has degree at least $2(n-1)$ and does not vanish at $(s,\dots,s,0)$. Hence,
 $$\bar \partial^{\mathbf m}\left({\prod_{1\leq j< k\leq n}(\bar w_j-\bar w_k)^2}\right)\bigg|_{(s,\dots,s,0)}\neq 0$$ if and only if $|\mathbf m|=(n-2)(n-1)$, i.e. $\mathbf m=(n-2,\dots,n-2)$. In this case, we have
 $$\prod_{j=1}^{n-1}\bar \partial_j^{n-2}\left({\prod_{1\leq j< k\leq n}(\bar w_j-\bar w_k)^2}\right)\bigg|_{(s,\dots,s,0)}=c_n\prod_{k=1}^{n-1}(\bar w_k-\bar w_n)^2|_{(s,\dots,s,0)}=c_ns^{2n-2}$$
 for some constant $c_n$.
Therefore, 
\begin{align}&\prod_{j=1}^{n-1} q_{n-2}(\mathcal D_{\bar w_j})\left(\frac{\prod_{1\leq j< k\leq n}(\bar w_j-\bar w_k)^2}{\pi^{n-1}\prod_{m=1}^n(1-z_m\bar w_m)\prod_{j,k=1}^n(1-z_j\bar w_k)}\right)\bigg|_{\bar w=(s,\dots,s,0)}\nonumber\\=&\sum_{\mathbf l\in\{0,1,\dots, n-2\}^{n-1}}d_{l_1}\cdots d_{l_{n-1}}\bar w^{\mathbf l}\bar \partial^{\mathbf l}\left(\frac{\prod_{1\leq j< k\leq n}(\bar w_j-\bar w_k)^2}{\pi^{n-1}\prod_{m=1}^n(1-z_m\bar w_m)\prod_{j,k=1}^n(1-z_j\bar w_k)}\right)\bigg|_{\bar w=(s,\dots,s,0)}\nonumber\\=&\left(\frac{d_{n-2}^{n-1}\prod_{j=1}^{n-1}\left(\bar w_j^{n-2}\bar \partial_j^{n-2}\right)\left({\prod_{1\leq j< k\leq n}(\bar w_j-\bar w_k)^2}\right)}{\pi^{n-1}\prod_{m=1}^n(1-z_m\bar w_m)\prod_{j,k=1}^n(1-z_j\bar w_k)}\right)\bigg|_{\bar w=(s,\dots,s,0)}\nonumber\\=&\frac{d^{n-1}_{n-2}c_ns^{n(n-1)}}{\prod_{m=1}^{n-1}(1-z_ms)\prod_{l=1}^n(1-z_ls)^{n-1}},\end{align}
which proves the claim (\ref{4.24}) and gives
\begin{equation}
	\tilde T^n(h_s)(z)=\sum_{\tau\in\mathcal S_n}\frac{d^{n-1}_{n-2}c_ns^{n(n-1)}}{\prod_{m=1}^{n-1}(1-\tau(z_m)s)\prod_{l=1}^n(1-z_ls)^{n-1}}.
\end{equation}
We next compute the norm of $\tilde T^n(h_s)$
\begin{align}
&\|	\tilde T^n(h_s)(z)\|^p_{L^p_{\text{sym}}(\mathbb D^n,|J_{\mathbb C}\Phi_n|^{2})}\nonumber\\=&\int_{\mathbb D^n}\left|\sum_{\tau\in \mathcal S_n}\frac{d^{n-1}_{n-2}c_ns^{n(n-1)}}{\prod_{m=1}^{n-1}(1-\tau(z_m)s)\prod_{l=1}^n(1-z_ls)^{n-1}}\right|^p\prod_{1\leq j<k\leq n }| z_j- z_k|^2dV(z)\nonumber\\=&\int_{\mathbb D^n}\frac{d^{p(n-1)}_{n-2}c_n^ps^{pn(n-1)}}{\prod_{l=1}^n|1-z_ls|^{p(n-1)}}\left|\sum_{\tau\in \mathcal S_n}\frac{1}{\prod_{m=1}^{n-1}(1-\tau(z_m)s)}\right|^p\prod_{1\leq j<k\leq n }| z_j- z_k|^2dV(z).
\end{align}
Set $$U_n(s)=\left\{w\in \mathbb D: \text{Arg}(1-ws)\in \left(-\frac{\pi}{6(n-1)},\frac{\pi}{6(n-1)}\right)\right\}.$$ Then for any $z=(z_1,\dots,z_n)\in (U_n(s))^n$ and $\tau\in\mathcal S_n$, 
\[\text{Arg}\left\{\frac{1}{\prod_{m=1}^{n-1}(1-\tau(z_m)s)}\right\}\in \left(-\frac{\pi}{6},\frac{\pi}{6}\right),\]
which yields that
\[\left|\sum_{\tau\in \mathcal S_n}\frac{1}{\prod_{m=1}^{n-1}(1-\tau(z_m)s)}\right|\gtrsim \frac{1}{\prod_{m=1}^{n-1}|1-z_ms|}.\]
Using this inequality, we have
\begin{align}\label{4.29}
	&\|	\tilde T^n(h_s)(z)\|^p_{L^p_{\text{sym}}(\mathbb D^n,|J_{\mathbb C}\Phi_n|^{2})}\nonumber\\=&\int_{\mathbb D^n}\frac{d^{p(n-1)}_{n-2}c_n^ps^{pn(n-1)}}{\prod_{l=1}^n|1-z_ls|^{p(n-1)}}\left|\sum_{\tau\in \mathcal S_n}\frac{1}{\prod_{m=1}^{n-1}(1-\tau(z_m)s)}\right|^p\prod_{1\leq j<k\leq n }| z_j- z_k|^2dV(z)\nonumber\\\gtrsim&\int_{(U_n(s))^n}\frac{1}{\prod_{l=1}^n|1-z_ls|^{p(n-1)}}\left|\sum_{\tau\in \mathcal S_n}\frac{1}{\prod_{m=1}^{n-1}(1-\tau(z_m)s)}\right|^p\prod_{1\leq j<k\leq n }| z_j- z_k|^2dV(z)\nonumber\\\gtrsim&\int_{(U_n(s))^n}\frac{\prod_{1\leq j<k\leq n }| z_j- z_k|^2}{\prod_{m=1}^{n-1}|1-z_ms|^p\prod_{l=1}^n|1-z_ls|^{p(n-1)}}dV(z)\nonumber\\=&\int_{(U_n(s))^n}\frac{\prod_{1\leq j<k\leq n }| z_j- z_k|^2}{\prod_{l=1}^n|1-z_ls|^{2(n-1)}}\cdot\frac{1}{\prod_{m=1}^{n-1}|1-z_ms|^p\prod_{l=1}^n|1-z_ls|^{p(n-1)-2(n-1)}}dV(z).
\end{align}
By a similar $(n-1)$-step partial fraction procedure, we obtain the following analogue of (\ref{4.202})
\[\frac{\prod_{1\leq j<k\leq n }(z_j- z_k)}{\prod_{l=1}^n(1-z_ls)^{n-1}}=\sum_{(l_1,\dots,l_n)\in\mathcal S_n}\frac{\text{sgn}((l_1,\dots,l_n))s^{-\frac{1}{2}n(n-1)}}{\prod_{t=1}^n(1-z_{l_t}s)^{n-t}}.\]
Hence (\ref{4.29}) becomes
\begin{align}
	&\|	\tilde T^n(h_s)(z)\|^p_{L^p_{\text{sym}}(\mathbb D^n,|J_{\mathbb C}\Phi_n|^{2})}\nonumber\\\gtrsim&\int_{(U_n(s))^n}\frac{\prod_{1\leq j<k\leq n }| z_j- z_k|^2}{\prod_{l=1}^n|1-z_ls|^{2(n-1)}}\cdot\frac{dV(z)}{\prod_{m=1}^{n-1}|1-z_ms|^p\prod_{l=1}^n|1-z_ls|^{p(n-1)-2(n-1)}}\nonumber\\\gtrsim&\int_{(U_n(s))^n}\left|\sum_{(l_1,\dots,l_n)\in\mathcal S_n}\frac{\text{sgn}((l_1,\dots,l_n))}{\prod_{t=1}^n(1-z_{l_t}s)^{n-t}}\right|^2\frac{dV(z)}{\prod_{m=1}^{n-1}|1-z_ms|^p\prod_{l=1}^n|1-z_ls|^{p(n-1)-2(n-1)}}.
\end{align}
We further restrict our region of integration to obtain more precise estimates. 
For $j\in\{1,\dots,n\}$ and $s\in(1-(5n!)^{-2n},1)$, we set  \[U_n(s,j)=U_n(s)\bigcap\left\{z:(5n!)^{2j}(1-s)<\left|z-\frac{1}{s}\right|<1\right\},\]
and set
$\mathbf U(s)=U_n(s,1)\times U_n(s,2)\times\cdots\times U_n(s,n).$
It is worth noting that we implement a positive lower bound $1-(5n!)^{-2n}$ for $s$ here so that $U_n(s,j)$ is nonempty and $\mathbf U(s)$ is asymmetric in its components. As the reader will see, we need this extra restriction for $s$ but not when $n=2$ since in higher dimensions, the desired integral estimates cannot be achieved solely by Forelli-Rudin estimate. The asymmetry of $\mathbf U(s)$  is also used.

By a polar coordinate computation,
\begin{align}\label{4.19}\int_{U_n(s,j)}\frac{dV(z)}{|1-zs|^k}=&s^{-k}\int_{U_n(s,j)}\frac{dV(z)}{|z-s^{-1}|^k}\nonumber\\=&s^{-k}\int_{-\frac{\pi}{6(n-1)}}^{\frac{\pi}{6(n-1)}}\int_{(5n!)^{2j}(1-s)}^{1}r^{1-k}drd\theta\nonumber\\=&\begin{cases}
	\frac{\pi}{3s^k(k-2)(n-1)}((5n!)^{2j(2-k)}(1-s)^{2-k}-1)& k>2\\-\frac{\pi}{3s^2(n-1)}(2j\log 5n!+\log(1-s))& k=2
\end{cases}.\end{align}
For functions $f(s)$ and $g(s)$, we write $f(s)\sim g(s)$ if
\[\lim_{s\to 1^-}\frac{f(s)}{g(s)}=1.\]
Then (\ref{4.19}) yields
\begin{align}\label{4.31}\int_{U_n(s,j)}\frac{dV(z)}{|1-zs|^k}\sim	\begin{cases}
\frac{\pi(5n!)^{2j(2-k)}(1-s)^{2-k}}{3s^k(k-2)(n-1)}& k>2\\ -\frac{\pi\log(1-s)}{3s^2(n-1)}& k=2
\end{cases}.\end{align}
Recall that for $\tau\in \mathcal S_n$, we let $\tau(j)$ be the index satisfying $z_{\tau(j)}=\tau(z_j)$. For  $p=\frac{2n}{n-1}$, the triangle inequality and Cauchy-Schwarz  inequality implies
\begin{align}
&\int_{(U_n(s))^n}\left|\sum_{\tau\in\mathcal S_n}\frac{\text{sgn}((l_1,\dots,l_n))}{\prod_{t=1}^n(1-z_{\tau(t)}s)^{n-t}}\right|^2\frac{1}{\prod_{m=1}^{n-1}|1-z_ms|^p\prod_{l=1}^n|1-z_ls|^{p(n-1)-2(n-1)}}dV(z)\nonumber\\\gtrsim&\int_{\mathbf U(s)}\left|\sum_{\tau\in\mathcal S_n}\frac{\text{sgn}((l_1,\dots,l_n))}{\prod_{t=1}^n(1-z_{\tau(t)}s)^{n-t}}\right|^2\frac{1}{\prod_{m=1}^{n-1}|1-z_ms|^p\prod_{l=1}^n|1-z_ls|^{p(n-1)-2(n-1)}}dV(z)\nonumber\\\gtrsim&\int_{\mathbf U(s)}\left(\frac{1}{\prod_{t=1}^n|1-z_{t}s|^{2n-2t}}-\left|\sum_{\substack{\tau\in\mathcal S_n\\\tau\neq I}}\frac{\text{sgn}((l_1,\dots,l_n))}{\prod_{t=1}^n(1-z_{\tau(t)}s)^{n-t}}\right|^2\right)\nonumber\frac{dV(z)}{\prod_{m=1}^{n-1}|1-z_ms|^{\frac{2n}{n-1}}\prod_{l=1}^n|1-z_ls|^{2}}\nonumber\\\gtrsim&\int_{\mathbf U(s)}\left(\frac{1}{\prod_{t=1}^n|1-z_{t}s|^{2n-2t}}-\sum_{\substack{\tau\in\mathcal S_n\\\tau\neq I}}\frac{n!}{\prod_{t=1}^n|1-z_{\tau(t)}s|^{2n-2t}}\right)\nonumber\frac{dV(z)}{\prod_{m=1}^{n-1}|1-z_ms|^{\frac{2n}{n-1}}\prod_{l=1}^n|1-z_ls|^{2}}\nonumber\\=&\int_{\mathbf U(s)}\left(\frac{1}{\prod_{t=1}^n|1-z_{t}s|^{2n-2t}}-\sum_{\substack{\tau\in\mathcal S_n\\\tau\neq I}}\frac{n!}{\prod_{t=1}^n|1-z_ts|^{2n-2\tau^{-1}(t)}}\right)\nonumber\frac{dV(z)}{\prod_{m=1}^{n-1}|1-z_ms|^{\frac{2n}{n-1}}\prod_{l=1}^n|1-z_ls|^{2}}.
\end{align}

 We claim that in the integral above, the first term will dominate the rest terms, and thus determines the size of the entire integral. We start by showing that the first term dominates the sum of those terms with $\tau^{-1}(n)\neq n$. Note that
\begin{align}\label{4.32}&\int_{\mathbf U(s)}\frac{dV(z)}{\prod_{t=1}^n|1-z_ts|^{2n-2\tau^{-1}(t)}\prod_{m=1}^{n-1}|1-z_ms|^{\frac{2n}{n-1}}\prod_{l=1}^n|1-z_ls|^{2}}\nonumber\\=&\int_{\mathbf U(s)}\frac{dV(z)}{|1-z_ns|^{2n+2-2\tau^{-1}(n)}\prod_{m=1}^{n-1}|1-z_ms|^{\frac{2n^2}{n-1}+2-2\tau^{-1}(m)}}\nonumber\\=&\int_{U_n(s,n)}\frac{dV(z_n)}{|1-z_ns|^{2n+2-2\tau^{-1}(n)}}\prod_{m=1}^{n-1}\int_{U_n(s,m)}\frac{dV(z_m)}{|1-z_ms|^{\frac{2n^2}{n-1}+2-2\tau^{-1}(m)}}.\end{align}
Since $1\leq m\leq n-1$, the  denominator factor $|1-z_ms|$ in (\ref{4.32}) has power strictly greater than 2. The factor $|1-z_ns|$ has power 2 only if $\tau^{-1}(n)=n$, or equivalently $\tau(z_n)=z_n$. By the Forelli-Rudin estimates (\ref{**}) and the fact that $\{\tau^{-1}(1),\dots,\tau^{-1}(n)\}=\{1,\dots,n\}$,
\begin{align}\label{4.241}&\int_{U_n(s,n)}\frac{dV(z_n)}{|1-z_ns|^{2n+2-2\tau^{-1}(n)}}\prod_{m=1}^{n-1}\int_{U_n(s,m)}\frac{dV(z_m)}{|1-z_ms|^{\frac{2n^2}{n-1}+2-2\tau^{-1}(m)}}\approx\begin{cases}
(1-s)^{-n^2-n}& \tau(n)\neq n\\-\frac{\log(1-s)}{(1-s)^{n^2+n}}& \tau(n)=n
\end{cases}.\end{align}
Thus for $s$ sufficiently close to $1$, the integral in (\ref{4.32}) with $\tau(n)=n$ dominates the ones with $\tau(n)\neq n$.
Therefore, we can further assume that
\[\int_{\mathbf U(s)}\left(\frac{1}{2}\frac{1}{\prod_{t=1}^n|1-z_{t}s|^{2n-2t}}-\sum_{\substack{\tau\in\mathcal S_n\\\tau(n)\neq n}}\frac{n!}{\prod_{t=1}^n|1-z_ts|^{2n-2\tau^{-1}(t)}}\right)\nonumber\frac{dV(z)}{\prod_{m=1}^{n-1}|1-z_ms|^{\frac{2n}{n-1}}\prod_{l=1}^n|1-z_ls|^2}\geq0,\]
which implies
\begin{align}\label{4.34}
	&\int_{\mathbf U(s)}\left(\frac{1}{\prod_{t=1}^n|1-z_{t}s|^{2n-2t}}-\sum_{\substack{\tau\in\mathcal S_n\\\tau\neq I}}\frac{n!}{\prod_{t=1}^n|1-z_ts|^{2n-2\tau^{-1}(t)}}\right)\nonumber\frac{dV(z)}{\prod_{m=1}^{n-1}|1-z_ms|^{\frac{2n}{n-1}}\prod_{l=1}^n|1-z_ls|^{2}}\\\gtrsim&\left(\frac{1}{2}\prod_{m=1}^{n-1}\int_{U_n(s,m)}\frac{dV(z_m)}{|1-z_ms|^{\frac{2n^2}{n-1}+2-2m}}-\sum_{\substack{\tau\in\mathcal S_{n-1}\\\tau\neq I}}\prod_{m=1}^{n-1}\int_{U_n(s,m)}\frac{n!dV(z_m)}{|1-z_ms|^{\frac{2n^2}{n-1}+2-2\tau^{-1}(m)}}\right)\int_{U_n(s,n)}\frac{dV(z_n)}{|1-z_ns|^{2}}.
\end{align}

Now we turn to show that the positive term in the last line of (\ref{4.34}) also dominates the rest terms. Since all these terms share the same $z_n$ part, the estimate (\ref{4.241}) is no longer able to distinguish one from another. Thus here, we will make use of the asymmetry of $\mathbf U(s)$ in $z_j$ variables to prove the claim. 
	
	By  (\ref{4.31}), we have
\begin{align}
\prod_{m=1}^{n-1}\int_{U_n(s,m)}\frac{dV(z_m)}{|1-z_ms|^{\frac{2n^2}{n-1}+2-2\tau^{-1}(m)}}&\sim\prod_{m=1}^{n-1}\frac{\pi(5n!)^{2m(2\tau^{-1}(m)-\frac{2n^2}{n-1})}(1-s)^{2\tau^{-1}(m)-\frac{2n^2}{n-1}}}{3s^{\frac{2n^2}{n-1}+2-2\tau^{-1}(m)}(\frac{2n^2}{n-1}-2\tau^{-1}(m))(n-1)}\nonumber\\&=\frac{\pi^{n-1}(1-s)^{-n^2-n}(5n!)^{-2n^3}}{3^{n-1}s^{n^2+n+2}(n-1)^{n-1}}\prod_{m=1}^{n-1}\frac{(5n!)^{4m\tau^{-1}(m)}}{(\frac{2n^2}{n-1}-2\tau^{-1}(m))}
\nonumber\\&=\frac{\pi^{n-1}(1-s)^{-n^2-n}(5n!)^{-2n^3}}{3^{n-1}s^{n^2+n+2}(n-1)^{n-1}}\frac{(5n!)^{4\sum_{m=1}^{n-1}m\tau^{-1}(m)}}{\prod_{m=1}^{n-1}(\frac{2n^2}{n-1}-2m)}.
\end{align}
Hence, for any permutation $\tau\in \mathcal S_{n-1}$ with $\tau\neq I$,
\begin{align}
\frac{\prod_{m=1}^{n-1}\int_{U_n(s,m)}\frac{dV(z_m)}{|1-z_ms|^{\frac{2n^2}{n-1}+2-2m}}}{\prod_{m=1}^{n-1}\int_{U_n(s,m)}\frac{dV(z_m)}{|1-z_ms|^{\frac{2n^2}{n-1}+2-2\tau^{-1}(m)}}}&\sim(5n!)^{4\sum_{m=1}^{n-1}(m^2-m\tau^{-1}(m))}\geq (4n!)^4.
\end{align}
Here $\sum_{m=1}^{n-1}(m^2-m\tau^{-1}(m))\geq 1$ follows by Cauchy-Schwarz inequality and the fact that the sum $\sum_{m=1}^{n-1}(m^2-m\tau^{-1}(m))$ is an integer.
Substituting these estimates into (\ref{4.34}), we finally obtain
\begin{align}
&\int_{U_n(s,n)}\frac{dV(z_n)}{|1-z_ns|^{2}}\left(\frac{1}{2}\prod_{m=1}^{n-1}\int_{U_n(s,m)}\frac{dV(z_m)}{|1-z_ms|^{\frac{2n^2}{n-1}+2-2m}}-\sum_{\tau\in\mathcal S_{n-1}}\prod_{m=1}^{n-1}\int_{U_n(s,m)}\frac{n!dV(z_m)}{|1-z_ms|^{\frac{2n^2}{n-1}+2-2\tau^{-1}(m)}}\right)\nonumber\\\gtrsim& \int_{U_n(s,n)}\frac{dV(z_n)}{|1-z_ns|^{2}}\prod_{m=1}^{n-1}\int_{U_n(s,m)}\frac{dV(z_m)}{|1-z_ms|^{\frac{2n^2}{n-1}+2-2m}}\left(\frac{1}{2}-\sum_{\tau\in\mathcal S_{n-1}}\frac{n!}{(4n!)^{4}}\right)
\nonumber\\\geq&\frac{1}{4} \int_{U_n(s,n)}\frac{dV(z_n)}{|1-z_ns|^{2}}\prod_{m=1}^{n-1}\int_{U_n(s,m)}\frac{dV(z_m)}{|1-z_ms|^{\frac{2n^2}{n-1}+2-2m}}\approx-(1-s)^{-n^2-n}\log(1-s),
\end{align}
which implies that $\|	\tilde T^n(h_s)\|^p_{L^p_{\text{sym}}(\mathbb D^n,|J_{\mathbb C}\Phi_n|^{2})}\gtrsim -(1-s)^{-n^2-n}\log(1-s)$.
Thus
\[\frac{\|	\tilde T^n(h_s)\|^p_{L^p_{\text{sym}}(\mathbb D^n,|J_{\mathbb C}\Phi_n|^{2})}}{\|h_s\|^p_{L^p_{\text{sym}}(\mathbb D^n,|J_{\mathbb C}\Phi_n|^{2})}}\gtrsim-\log(1-s)\to \infty\]
as $s\to 1$, proving that  $\tilde T^n$ is unbounded on $L^p_{\text{sym}}(\mathbb D^n,|J_{\mathbb C}\Phi_n|^{2})$ for $p=\frac{2n}{n-1}$.
\end{proof}

\section{Some remarks}
\paragraph{1} In \cite{ZhenghuiWick3}, we studied weak-type estimates of the Bergman projection on the Hartogs triangle and showed the projection is of weak-type $(4,4)$ but not of weak-type $(\frac{4}{3},\frac{4}{3})$.  These results together with the Marcinkiewicz interpolation also recover the sharp $L^p$ regular range $(\frac{4}{3},4)$ for the projection on the Hartogs triangle.  Similarly, weak-type $(p,p)$ estimates of $P_{\mathbb G^n}$ when $p=\frac{2n}{n\pm1}$ could lead to an alternative approach for Theorem \ref{main}.
\vskip 10pt
\paragraph{2} In \cite{CJY}, Chen, Jin and Yuan obtained the Sobolev $L^p$ boundedness for $P_{\mathbb G}$ from $W^{k,p}(\mathbb G)$ to some weighted $W^{k,p}$ spaces for $p>2$. With $L^p$ irregularity results obtained for $P_{\mathbb G^n}$, it would be interesting to investigate the $W^{k,p}$ (ir)regularity for $P_{\mathbb G^n}$. In addition to estimates for  $P_{\mathbb G^n}$, one may further consider $L^p$ boundedness and compactness of operators that are related to the Bergman projection, such as Toeplitz operators and Hankel operators.
\vskip 10pt
\paragraph{3} The symmetrized polydisc $\mathbb G^n$ can be viewed as the quotient domain $\mathbb D^n/\mathcal S_n$ where $\mathcal S_n$ is the group of permutations of variables acting on $\mathbb D^n$.  It is interesting to see whether our method can be generalized to obtain similar results on other quotient domains of $\mathbb D^n$. For instance, the $L^p$ norm of $P_{\mathbb G^n}$ is equivalent to the $L^p$ norm of $P_{\mathbb D^n}$ over ${L^p_{\text{anti}}(\mathbb D^n,|J_{\mathbb C}\Phi_n|^{2-p})}$, a subspace of $L^p(\mathbb D^n,|J_{\mathbb C}\Phi_n|^{2-p})$ that is related to $\mathcal S_n$. On this subspace, we are able to construct the operator $T^n_2$ which equals $P_{\mathbb D^n}$.  It is interesting to see if such a proving strategy can be abstracted to work for general quotient domains.

\bibliographystyle{alpha}
\bibliography{2} 
\end{document}